\documentclass[11pt]{article}
\usepackage{eurosym}
\usepackage{amsfonts}
\usepackage{amssymb}
\usepackage{amsthm}
\usepackage{amsmath}
\usepackage{graphicx}
\usepackage{empheq}
\usepackage{indentfirst}
\usepackage{cite}
\usepackage{mathrsfs}
\usepackage{cases}
\usepackage{graphics}
\usepackage{xcolor}
\usepackage{bm}
\usepackage{makeidx}
\usepackage[T1]{fontenc}
\usepackage{times}

\newtheorem{theorem}{\textbf{Theorem}}[section]
\newtheorem{lemma}{\textbf{Lemma}}[section]
\newtheorem{proposition}{\textbf{Proposition}}[section]
\newtheorem{corollary}{\textbf{Corollary}}[section]
\newtheorem{remark}{\textbf{Remark}}[section]
\newtheorem{definition}{\textbf{Definition}}[section]

\allowdisplaybreaks[4]

\def\be{\begin{equation}}
\def\ee{\end{equation}}
\def\bea{\begin{eqnarray}}
\def\eea{\end{eqnarray}}
\def\bt{\begin{theorem}}
\def\et{\end{theorem}}
\def\bl{\begin{lemma}}
\def\el{\end{lemma}}
\def\br{\begin{remark}}
\def\er{\end{remark}}
\def\bp{\begin{proposition}}
\def\ep{\end{proposition}}
\def\bc{\begin{corollary}}
\def\ec{\end{corollary}}
\def\bd{\begin{definition}}
\def\ed{\end{definition}}
\def\vp{\varphi}

\begin{document}

\title{Separation property and asymptotic behavior for a transmission problem of  the bulk-surface coupled Cahn-Hilliard system with singular potentials and its Robin approximation}
\author{
Maoyin Lv \thanks{%
School of Mathematical Sciences, Fudan University, Shanghai
200433, P.R. China. Email: \texttt{mylv22@m.fudan.edu.cn} }, \ \
Hao Wu \thanks{%
Corresponding author. School of Mathematical Sciences and Shanghai Key Laboratory for Contemporary Applied Mathematics, Fudan University, Shanghai 200433, P.R. China. Email: \texttt{haowufd@fudan.edu.cn} } }
\date{\today }
\maketitle

\begin{abstract}
\noindent
We consider a class of bulk-surface coupled Cahn-Hilliard systems in a smooth, bounded domain $\Omega\subset\mathbb{R}^{d}$ $(d\in\{2,3\})$, where the trace value of the bulk phase variable is connected to the surface phase variable via a Dirichlet boundary condition or its Robin approximation.
For a general class of singular potentials (including the physically relevant logarithmic potential), we establish the regularity propagation of global weak solutions to the initial boundary value problem. In particular, when the spatial dimension is two, we prove the instantaneous strict separation property, which ensures that every global weak solution remains uniformly away from the pure states $\pm 1$ after any given positive time. In the three-dimensional case, we obtain the eventual strict separation property that holds for sufficiently large time. This strict separation property allows us to prove that every global weak solution converges to a single equilibrium as time goes to infinity using the {\L}ojasiewicz-Simon approach. Finally, we study the double obstacle limit for the problem with logarithmic potentials in the bulk and on the boundary, showing that as the absolute temperature $\Theta$ tends to zero, the corresponding weak solutions converge (for a suitable subsequence) to a weak solution of the problem with a double obstacle potential. \medskip

 \smallskip

\end{abstract}

%\tableofcontents

\section{Introduction}

The Cahn-Hilliard equation is a fundamental diffuse-interface model proposed in \cite{CH} to describe spinodal decomposition of binary alloys.
In the diffuse-interface framework, an order parameter $\varphi$, known as the phase-field, is introduced to characterize the difference between local concentrations (e.g., volume fractions) of two components in a binary mixture. Regions of pure phases correspond to areas with $\varphi=\pm 1$. They are separated by a thin interfacial layer (diffuse interface) whose thickness is proportional to a small parameter $\varepsilon>0$. In intermediate regions, the phase function $\varphi$ exhibits a continuous transition between $-1$ and $1$. Diffuse-interface models, including the Cahn-Hilliard equation and its variants, provide an efficient tool for studying morphological changes of free interfaces in fluid or solid mixtures from both theoretical and numerical perspectives \cite{Mi}. In recent years, they have been successfully applied to describe phase separation phenomena arising in various areas of scientific research, such as diblock copolymers, image painting, tumor growth and two-phase flows.

Let us consider the following Cahn-Hilliard equation
\begin{align}
	\left\{
	\begin{array}{ll}
		\partial_{t}\varphi=\Delta\mu,&\quad \text{in }\Omega\times (0,+\infty),\\
		\mu=-\Delta\varphi+F'(\varphi),&\quad \text{in }\Omega\times (0,+\infty),
	\end{array}\right.
\label{CH}
\end{align}
where $\Omega\subset\mathbb{R}^{d}$ $(d\in\{2,3\})$ is a bounded domain with smooth boundary $\Gamma:=\partial\Omega$. For simplicity, several physical parameters such as the interfacial thickness, interfacial tension and mobility etc have been set to unity since their values do not have any influence on the subsequent analysis. The functions $\varphi:\Omega\times (0,+\infty)\rightarrow[-1,1]$ and $\mu:\Omega\times (0,+\infty)\to\mathbb{R}$ denote the phase-field and chemical potential in the bulk, respectively.
To ensure well-posedness of the evolution equation \eqref{CH}, appropriate boundary conditions (accompanied by an initial condition for $\varphi$) must be taken into account. Classical choices include the homogeneous Neumann boundary conditions for the phase-field variable and the chemical potential, specifically,  	$\partial_{\mathbf{n}}\varphi=\partial_{\mathbf{n}}\mu=0$ on $\Gamma\times (0,+\infty)$. The resulting initial-boundary value problem has been extensively studied in the literature (see, e.g., \cite{AW,EZ,GGM,KNP,Mi,RH,MZ04} and the references therein). Apart from their crucial role in the mathematical analysis, the physics associated with these boundary conditions is also significant for real-world applications. For instance, $\partial_{\mathbf{n}}\mu=0$ represents a no-flux boundary condition, indicating that there is no exchange of mass between the inside and outside of the domain $\Omega$. However, the boundary condition $\partial_{\mathbf{n}}\varphi=0$ may have limitations from a physical perspective, as it implies that the diffuse interface between two components of the mixture intersects
the boundary $\Gamma$ with a fixed contact angle of ninety degrees at all times, which may be unrealistic (cf.\cite{QWS}). Generally speaking, the physics associated with boundary conditions cannot be simply deduced from that associated with evolution equations in the bulk. The coexistence of different dissipative processes in the bulk and on the boundary is a common phenomenon in mixtures of materials (see \cite{FMD,KEMRSBD,QWS}).

In recent years, the study of boundary effects in phase separation processes of binary mixtures has garnered significant attention. To describe short-range interactions between the solid boundary and the mixture contained within, several types of dynamic boundary conditions for the Cahn-Hilliard equation have been introduced and analyzed in the literature, as seen in the recent review paper \cite{W} and the references therein. In this work, we specifically focus on the following boundary conditions:
\begin{align}
	\left\{
	\begin{array}{ll}
        \partial_{\mathbf{n}}\mu=0,&\quad \text{on }\Gamma\times (0,+\infty),\\
		K\partial_{\mathbf{n}}\varphi=\psi-\varphi,&\quad \text{on }\Gamma\times (0,+\infty),\\
		\partial_{t}\psi=\Delta_{\Gamma}\theta,&\quad \text{on }\Gamma\times (0,+\infty),\\
		\theta=\partial_{\mathbf{n}}\varphi-\Delta_{\Gamma}\psi+G'(\psi),&\quad \text{on }\Gamma\times (0,+\infty),
	\end{array}\right.
\label{dynamic}
\end{align}
where $\mathbf{n}:=\mathbf{n}(x)$ stands for the unit outer normal vector on $\Gamma$.
The symbols $\partial_{\mathbf{n}}$ and $\Delta_{\Gamma}$  denote the outward normal derivative and the Laplace-Beltrami operator on the boundary, respectively.
In \eqref{dynamic}, the surface phase-field $\psi:\Gamma\times (0,+\infty)\to[-1,1]$ represents distribution of the binary mixture on the boundary and $\theta: \Gamma\times (0,+\infty)\to\mathbb{R}$ denotes the surface chemical potential. To solve the evolution problem \eqref{CH}--\eqref{dynamic}, we also impose the initial conditions:
\begin{align}
	\varphi|_{t=0}=\varphi_{0}\quad \text{in }\Omega,\qquad\psi|_{t=0}=\psi_{0}\quad \text{on }\Gamma.
\label{initial}
\end{align}
In \eqref{dynamic}, we neglect mass transfer between the bulk and boundary, maintaining the no-flux boundary condition $\partial_{\mathbf{n}}\mu=0$ (cf. \cite{GMS,KLLM}, where possible adsorption or desorption processes between the materials in the bulk and on the boundary were considered, see \eqref{kllm} below). This condition also implies that the chemical potentials $\mu$ and $\theta$ are not directly coupled, and interactions between the bulk and surface materials occur through the phase-fields $\varphi$ and $\psi$. The third and fourth conditions in \eqref{dynamic} yield a surface Cahn-Hilliard equation for $\psi$ on $\Gamma$, which is coupled to the bulk through the normal derivative $\partial_\mathbf{n} \varphi$. Finally, we observe that the bulk and surface phase-field functions are coupled through the second condition in \eqref{dynamic} with a parameter $K\in[0,+\infty)$. When $K=0$, it simplifies to a transmission condition $\varphi|_{\Gamma}=\psi$ on $\Gamma\times(0,+\infty)$, i.e., a (non-homogeneous) Dirichlet boundary condition for the bulk phase-field, where $\varphi|_{\Gamma}$ represents the trace of $\varphi$ on the boundary. When $K\in (0,+\infty)$, the corresponding condition provides a Robin type approximation (also known as the boundary penalty method), see \cite{CFL,KL20} for further discussions. Formally, in the limit $K\to +\infty$, we find $\partial_{\mathbf{n}}\varphi=0$, which, together with $\partial_{\mathbf{n}}\mu=0$, indicates that the dynamics of $\varphi$ in the bulk and the dynamics of $\psi$ on the boundary become independent. This situation is less interesting and will not be considered  here.

The set of boundary conditions \eqref{dynamic} with $K=0$ was initially derived in \cite{LW}, employing an energetic variational approach that integrates the least action principle and Onsager's principle of maximum energy dissipation. It characterizes a specific phase separation process involving transmission dynamics between the bulk and the boundary,  inherently satisfying three physical properties: mass conservation and force balance both in the bulk and on the boundary, as well as dissipation of the total free energy (see \eqref{massconservation}, \eqref{BEL-1} below). Extensions have been made in \cite{KL20} to the case of an affine linear transmission condition, $\varphi|_\Gamma=\alpha \psi+\beta$, further to a more general scenario,  $\varphi|_\Gamma=H(\psi)$, for some continuous function $H:\mathbb{R}\to \mathbb{R}$. Additionally, the case with $K\in (0,+\infty)$ was proposed as a Robin approximation of the Dirichlet-type transmission condition, and the convergence as $K\to 0$ for the affine linear case was rigorously justified with an error estimate.  It is noteworthy that the general transmission condition, such as like $\varphi|_\Gamma=H(\psi)$, was first considered in \cite{CFL} for the Allen-Cahn system to account for some intriguing and non-trivial couplings between bulk and surface dynamics. In this study, we confine ourselves to the linear case, $\varphi|_\Gamma=\psi$, for simplicity.

Let us now present some important properties of the bulk-surface coupled system \eqref{CH}--\eqref{initial}. For sufficiently regular solutions, we find the conservation of mass both in the bulk and on the boundary:
\begin{align}
	\int_{\Omega}\varphi(t)\,\mathrm{d}x=\int_{\Omega}\varphi_{0}\,\mathrm{d}x,\quad
	\int_{\Gamma}\psi(t)\,\mathrm{d}S=\int_{\Gamma}\psi_{0}\,\mathrm{d}S, \qquad\forall\,t\in[0,+\infty). \label{massconservation}
\end{align}
Next, the total free energy associated to the system \eqref{CH}--\eqref{dynamic} is given by
\begin{align}
	E\big(\varphi,\psi\big):=\underbrace{\int_{\Omega}\Big(\frac{1}{2}|\nabla\varphi|^{2}
	+F(\varphi)\Big)\,\mathrm{d}x}_{\text{bulk free energy}}+\underbrace{\int_{\Gamma}\Big(\frac{1}{2}|\nabla_{\Gamma}\psi|^{2}
	+G(\psi)\Big)\,\mathrm{d}S}_{\text{surface free energy}}+\frac{\chi(K)}{2}\int_{\Gamma}|\psi-\varphi|^{2}\,\mathrm{d}S,\notag%\label{Freeenergy}
\end{align}
where $\nabla$ and $\nabla_{\Gamma}$ denote the gradient operator in $\Omega$ and the tangential (surface) gradient operator on $\Gamma$, respectively. Here we set
\begin{align*}
	\chi(K):=
	\begin{cases}
	0,&\quad \text{if }K=0,\smallskip \\
	\dfrac{1}{K},&\quad \text{if }K\in(0,+\infty),
	\end{cases}
\end{align*}
to distinguish the case of Dirichlet transmission condition and its Robin approximation. The first and second terms in $E$ correspond to the bulk and surface free energies of Ginzburg-Landau type, while the third term measures the deviation of the trace $\varphi|_{\Gamma}$ from $\psi$ (cf. \cite{CFL,KL20}). The surface Dirichlet energy $(1/2)\int_{\Gamma} |\nabla_{\Gamma}\psi|^{2}\mathrm{d}S$ corresponds to possible surface diffusion and yields a regularizing effect on the boundary (cf. \cite{CFSJEE,LW} for the case $K=0$). A direct calculation yields that, for sufficiently regular solutions to problem \eqref{CH}--\eqref{initial}, the following energy identity is fulfilled:
\begin{align}
	\frac{\mathrm{d}}{\mathrm{d}t}E\big(\varphi(t),\psi(t)\big) +\int_{\Omega}|\nabla\mu(t)|^{2}\,\mathrm{d}x	+\int_{\Gamma}|\nabla_{\Gamma}\theta(t)|^{2}\,\mathrm{d}S=0, \qquad\forall\,t\in(0,+\infty).
\label{BEL-1}
\end{align}
We note that the bulk and boundary chemical potentials $\mu$, $\theta$ can be obtained from the variation of the total free energy. Moreover, the Cahn-Hilliard system \eqref{CH}--\eqref{dynamic} can be regarded as a gradient flow of  $E$ with respect to a suitable inner product (see \cite{GK,KL20}).

The nonlinear functions $F$ and $G$ represent the homogeneous free energy densities in the bulk and on the boundary, respectively. We shall treat general singular potentials in a setting similar to that in \cite{CFW}, namely, with the decomposition
\begin{equation}
F=\widehat{\beta}+\widehat{\pi},\qquad G=\widehat{\beta}_{\Gamma}+\widehat{\pi}_{\Gamma},
\label{decom}
\end{equation}
where $\widehat{\beta}$, $\widehat{\beta}_{\Gamma}$ are proper convex lower semicontinuous functions and $\widehat{\pi}$, $\widehat{\pi}_{\Gamma}$ are smooth concave perturbations.
Then we denote $F'=\beta+\pi$ and $G'=\beta_{\Gamma}+\pi_{\Gamma}$, where $\beta=\partial\widehat{\beta}$, $\beta_{\Gamma}=\partial\widehat{\beta}_{\Gamma}$ are the  subdifferentials of $\widehat{\beta}$, $\widehat{\beta}_{\Gamma}$
and $\pi=\widehat{\pi}'$, $\pi_{\Gamma}=\widehat{\pi}_{\Gamma}'$ are usual derivatives.
In applications related to materials science, a physically relevant choice for $F$ (and $G$) is the logarithmic potential \cite{CH} (also referred to as the Flory-Huggins potential):
\begin{align}	W_{\text{log}}(r):=\frac{\Theta}{2}\underbrace{\big[(1+r)\text{ln}(1+r)+(1-r)\text{ln}(1-r)\big]}_{=:F_0(r)} -\frac{\Theta_{c}}{2}r^{2},\qquad r\in(-1,1),
\label{logarithmic}
\end{align}
where $\Theta>0$ is the absolute temperature of the mixture and $\Theta_{c}$ is the critical temperature for phase separation. We find that the logarithmic part $F_{0}\in C([-1,1])\cap C^{\infty}(-1,1)$ is convex, while $W_{\text{log}}$ is non-convex with a double-well structure if $\Theta_{c}>\Theta$. $W_{\text{log}}$ is referred to as a singular potential since the derivative $f_0(r)=F_0'(r)$  diverges to $\pm \infty$ as $r\to\pm 1$. In practice, the logarithmic potential is often approximated by a regular potential of polynomial type like
$$
W_{\text{reg}}(r)=\frac{1}{4}(r^2-1)^2,\qquad r\in \mathbb{R}.
$$
We also mention another commonly used singular potential, that is,  the so-called double-obstacle potential  (see \cite{BE}):
\begin{align}
	W_{\text{2obs}}(r)=I_{[-1,1]}(r)-\dfrac{\Theta_{c}}{2}r^{2}=
	\left\{
	\begin{array}{ll}
		-\dfrac{\Theta_{c}}{2}r^{2},&\text{if }r\in[-1,1],\smallskip \\
		+\infty,&\text{else},
	\end{array}\right.
\label{2obs}
\end{align}
where $I_{[-1,1]}(r)$ is the indicator function of $[-1,1]$. Then it holds
 $W'_{\text{2obs}}(r)=\partial I_{[-1,1]}(r)-\Theta_c r$.

We first recall some related results in the case where $K=0$. The problem \eqref{CH}--\eqref{initial} with regular potentials $F$, $G$ was initially analyzed in \cite{LW}, where well-posedness and long-time behavior (i.e., convergence to a single equilibrium) of global weak/strong solutions were established. Subsequently, by introducing a slightly weaker notation of the solution, the authors of \cite{GK} proved existence and uniqueness of weak solutions via a gradient flow approach, removing the additional geometric assumption imposed in \cite{LW} when the surface diffusion is absent. Besides, the existence of a global attractor and exponential attractors were obtained in \cite{MW}. The case with singular potentials (including \eqref{logarithmic}, \eqref{2obs}) is more intricate. Existence and uniqueness of weak/strong solutions were proved in \cite{CFW} based on a novel time-discretization scheme for a regularized problem with viscous terms (in the chemical potentials) and Yosida's approximation (for the singular nonlinearities $F$, $G$). Regularity propagation of weak solutions and the existence of a global attractor were proved in \cite{MW}. However, less is known about the long-time behavior in this case (for instance, convergence to a single equilibrium and the existence of exponential attractors), since further regularity properties of weak solutions, particularly whether the solution will stay uniformly away from the pure states, remain unclear. We also mention \cite{CFSJEE}, in which the authors investigated the asymptotic limit as the surface diffusion acting on the boundary phase-field variable vanishes. They obtained a forward-backward dynamic boundary condition at the limit, and, thanks to the Dirichlet type transmission condition (i.e., $K=0$), they were able to prove well-posedness of the limit problem with a general class of singular potentials.

Next, when $K\in(0,+\infty)$, problem \eqref{CH}--\eqref{initial} with regular potentials $F$, $G$ was investigated in \cite{KL20}. Well-posedness and the asymptotic limit as $K\to 0$, both with or without surface diffusion on the boundary, were established. Recently,  in \cite{KS24}, a general class of bulk-surface Cahn-Hilliard systems with convection, dynamic boundary conditions, and regular potentials was analyzed. The authors considered the following boundary condition that extends the second one in \eqref{dynamic}:
\begin{align}
	L\partial_{\mathbf{n}}\mu=\theta-\mu,\quad\text{on }\Gamma\times(0,+\infty), \quad \text{with}\ L\in [0,+\infty].
\label{kllm}
\end{align}
 They first proved existence of global weak solutions in the case $K, L \in (0,+\infty)$ using a suitable Faedo-Galerkin approximation, and then obtained the existence of weak solutions for all other cases through the asymptotic limits (i.e., letting $K$ and $L$ tend to $0$ or to $+\infty$). Unfortunately, singular potentials like the logarithmic potential \eqref{logarithmic} or the double-obstacle potential \eqref{2obs} are not admissible in this context (see \cite[Remark 2.1]{KS24}).
 For problem \eqref{CH}--\eqref{initial} with $K\in(0,+\infty)$ and singular potentials, the only available analytic result was given in \cite{CKSS}, where the authors studied the Cahn-Hilliard system \eqref{CH}--\eqref{dynamic} coupled to a Brinkman equation that describes the motion of creeping two-phase flows in a porous medium. They established the existence of global weak solutions for singular potentials, including \eqref{logarithmic} and \eqref{2obs}. However, due to the coupling with a velocity equation, the uniqueness and regularity of weak solutions remain open. Furthermore, in contrast to \cite{CFSJEE}, the dynamics in the bulk do not seem sufficiently strong to compensate the backward dynamics in the limit of vanishing surface diffusion on the boundary, since $K\in (0,+\infty)$ implies a weaker relationship between $\varphi|_\Gamma$ and $\psi$.

Our aim in this study is to explore the strict separation property and asymptotic behavior of global weak solutions to problem \eqref{CH}--\eqref{initial} with singular potentials for all $K\in[0,+\infty)$. The findings are summarized as follows:
\begin{itemize}
	\item[(1)] \emph{Strict separation from pure states}. In three dimensions, we establish the eventual strict separation property for a general class of singular potentials, ensuring that every global weak solution stays uniformly distant from the pure states $\pm 1$ after a sufficiently large time. Our proof relies on the gradient flow structure of problem \eqref{CH}--\eqref{initial} and the strict separation property of the $\omega$-limit set (see Theorem \ref{eventual}). When the spatial dimension is two, under certain additional assumptions (see $(\mathbf{A5a})$, $(\mathbf{A5b})$ in Section 2), we prove the instantaneous strict separation property, meaning that every global weak solution stays uniformly away from $\pm 1$ after an arbitrary given positive time (see Theorem \ref{intantaneous}). We successfully extend the direct method from \cite{GGG} and the De Giorgi's iteration scheme from \cite{GP}, originally applied to the Cahn-Hilliard equation with  homogeneous Neumann boundary conditions, to the present case involving non-trivial bulk-surface interactions.
\item[(2)] \emph{Long-time behavior}. Thanks to the strict separation property, we can regard the singular potentials as globally Lipschitz functions on a compact subset of $(-1,1)$. Consequently, problem \eqref{CH}--\eqref{initial} can be approached as the case with regular potentials (cf. \cite{AW,GW,RH}). More precisely, assuming that $F$ and $G$ are real analytic on $(-1,1)$, we can prove that every global weak solution converges to a single equilibrium as $t\rightarrow +\infty$ (see Theorem \ref{equilibrium}).  The proof is aided by an extended {\L}ojasiewicz-Simon type gradient inequality that incorporates bulk-surface interactions (see Lemma \ref{LS}).
\item[(3)] \emph{Double obstacle limit}. Consider problem \eqref{CH}--\eqref{initial} with $F=G=W_{\text{log}}$ in a finite time interval $[0,T]$ for any given final time $T\in(0,+\infty)$.
Despite the bulk-surface coupling structure, we prove that weak solutions $(\varphi_{\Theta}, \psi_{\Theta}, \mu_{\Theta},\theta_{\Theta})$ corresponding to the parameter $\Theta\in(0,1]$ converge (for a suitable subsequence) to the weak solution of the limit system as $\Theta\rightarrow 0$. In this limit system, equations \eqref{CH}$_{2}$, \eqref{dynamic}$_{4}$ for the bulk and surface chemical potentials are replaced by differential inclusions related to the subgradient of the double obstacle potential $W_{\text{2obs}}$ (see Theorem \ref{doubleobstacle}). The proof relies on uniform estimates with respect to the parameter $\Theta\in(0,1]$ that can be derived from the energy equality \eqref{BEL-1} and a compactness argument similar to that in \cite{Abels11}.
\end{itemize}

The strict separation property plays a pivotal role in the study of phase-field models with singular potentials. It simplifies the handling of singular potentials and enables us to gain further insights into the regularity and long-time behavior of global solutions.
For the Cahn-Hilliard equation \eqref{CH} subject to homogeneous Neumann boundary conditions, the eventual separation property was proven in \cite{AW} using a dynamic approach (see \cite{GP} for an alternative proof based on De Giorgi's iteration scheme). Regarding the eventual separation property for the Cahn-Hilliard equation with Cahn-Hilliard type dynamic boundary conditions, we refer to the recent works \cite{FW} (for the case $K=L=0$) and \cite{LvWu-2} (for the case $K=0$, $L\in[0,+\infty)$).
On the other hand, the instantaneous strict separation property is more intricate since it depends on the spatial dimension and necessitates additional assumptions on the singular potential. Taking a singular potential $F=\widehat{\beta}+\widehat{\pi}$ as an example, a commonly used assumption is the following pointwise relation between the first and second order derivatives $\widehat{\beta}'$, $\widehat{\beta}''$:
\begin{align}
	\widehat{\beta}''(r)\leq C_{\sharp}e^{C_{\sharp}|\widehat{\beta}'(r)|^{\kappa_{\sharp}}}, \qquad\forall\,r\in(-1,1),\label{pointwise}
\end{align}
for some constants $C_{\sharp}>0$ and $\kappa_{\sharp}\in[1,2)$. See  \cite{GGG,GGM,HW21,MZ04,W} and the references therein for detailed discussions in the case of homogeneous Neumann boundary conditions.
Recently, with the aid of a suitable De Giorgi's iteration scheme, the authors of \cite{GP} were able to treat a wilder set of singular potentials based on some milder growth condition just for the first-order derivative $\widehat{\beta}'$ near the pure states $\pm1$, that is , as $\delta\to0$, there exists $\kappa>1/2$, such that
	\begin{align}
	\frac{1}{\widehat{\beta}'(1-2\delta)}
=O\Big(\frac{1}{|\ln \delta|^{\kappa}}\Big),\qquad	\frac{1}{|\widehat{\beta}'(1-2\delta)|} =O\Big(\frac{1}{|\ln \delta|^{\kappa}}\Big).
\label{milder}
\end{align}
Regarding the Cahn-Hilliard equation coupled with Cahn-Hilliard type dynamic boundary conditions, the authors of \cite{FW} achieved instantaneous separation in two dimensions under the assumption \eqref{pointwise} with $\kappa_{\sharp}=1$ for the case $K=L=0$. Recently, in \cite{LvWu-2}, this result was extended to the case $K=0$, $L\in[0,+\infty)$ using a suitable De Giorgi's iteration scheme under the assumption \eqref{milder}. To the best of our knowledge, for the bulk-surface coupled Cahn-Hilliard system in the case $K\in(0,+\infty)$, whether the strict separation property holds or not remains an open question. Lastly, it is worth mentioning recent progresses on the strict separation property of nonlocal and fractional Cahn-Hilliard equations, as seen in \cite{GGG,GP,Gi,Po} and the references therein.

Concerning the double obstacle limit, the Cahn-Hilliard equation \eqref{CH} with the double obstacle potential \eqref{2obs} and homogenous Neumann boundary conditions was first studied in \cite{BE}, and later it was proven in \cite{EL} that as $\Theta\rightarrow0$, solutions of the Cahn-Hilliard equation with $W_{\text{log}}$ converge to solutions corresponding to $W_{\text{2obs}}$.
This type of result is also known as the deep quench limit in the literature. For extensions to the Navier-Stokes-Cahn-Hilliard system for incompressible two-phase flows, see \cite{Abels11} for the case with matched densities and \cite{AGG23} for the case with unmatched densities. Returning to our problem \eqref{CH}--\eqref{initial}, well-posedness in the case with $F=G=W_{\text{2obs}}$ and $K=0$ was established in \cite{CFW}, while for the case $K\in(0,+\infty)$, existence of weak solutions was obtained in \cite{CKSS}. In this study, we rigorously justify the limiting procedure as $\Theta\to 0$, with the novelty being the treatment of the bulk-surface coupling.

In recent years, the Cahn-Hilliard equation subject to Cahn-Hilliard type dynamic boundary conditions for the phase-field variable and the general boundary condition \eqref{kllm} for the chemical potential has also garnered significant attention. The parameter $L$ distinguishes different types of adsorption or desorption processes between the materials in the bulk and on the boundary, and the value of $1/L$ can be interpreted as a kinetic rate \cite{KLLM}. Related models with $L\in [0,+\infty]$ have been extensively studied in the literature from various viewpoints, as seen in \cite{CF15,CFS,CFSJEE,FW,GK,GKY,GMS,KLLM,KS24,LW,LvWu,LvWu-2,MW}. For further discussions on this topic, we refer to the review paper \cite{W}.

The remainder of this paper is structured as follows: In Section 2, we introduce notations, assumptions and preliminaries, followed by the presentation of our main results. In Section 3, we present the well-posedness of problem \eqref{model} and establish some basic properties for global weak solutions, including the mass conservation law, energy dissipation and global regularity. In Section 4, we first demonstrate the instantaneous strict separation property in two dimensions and then establish the eventual strict separation property that applies in both two and three dimensions. In Section 5, we show that every global weak solution converges to a single equilibrium as time goes to infinity, using the {\L}ojasiewicz-Simon approach. Section 6 is dedicated  to proving the double obstacle limit as $\Theta\to 0$. In the Appendix, we provide useful tools and outline the proof of a generalized {\L}ojasiewicz-Simon inequality.

\section{Main Results}
\setcounter{equation}{0}
\subsection{Preliminaries}
For any real Banach space $X$, we denote its norm by $\|\cdot\|_X$, its dual space by $X'$
and the duality pairing between $X'$ and $X$ by
$\langle\cdot,\cdot\rangle_{X',X}$. If $X$ is a Hilbert space,
we denote the associated inner product by $(\cdot,\cdot)_X$.
The space $L^q(0,T;X)$ ($1\leq q\leq +\infty$)
stands for the set of all strongly measurable $q$-integrable functions with
values in $X$, or, if $q=+\infty$, essentially bounded functions. The space $L^q_{\mathrm{uloc}}(0,+\infty;X)$ denotes the uniformly local variant of $L^q(0,+\infty;X)$ consisting of all strongly measurable $f:[0,+\infty)\to X$ such that
\begin{equation*}
	\|f\|_{L^q_{\mathrm{uloc}}(0,+\infty;X)}:= \mathop\mathrm{sup}\limits_{t\ge0} \|f\|_{L^q(t,t+1;X)} <+ \infty.
\end{equation*}
If $T\in(0,+\infty)$, we simply have $L^q_{\mathrm{uloc}}(0,T;X)=L^q(0,T;X)$.
The space $C([0,T];X)$ denotes the Banach space of all bounded and
continuous functions $u:[ 0,T] \rightarrow X$ equipped with the supremum
norm, while $C_{w}([0,T];X)$ denotes the topological vector space of all
bounded and weakly continuous functions.

Let $\Omega$ be a bounded domain in $\mathbb{R}^d$ ($d\in \{2,3\}$) with sufficiently smooth boundary $\Gamma:=\partial \Omega$.
We denote by $|\Omega|$ and $|\Gamma|$
the Lebesgue measure of $\Omega$ and the Hausdorff measure of $\Gamma$, respectively.
For any $1\leq q\leq +\infty$, $k\in \mathbb{N}$, the standard Lebesgue and Sobolev spaces on $\Omega$ are denoted by $L^{q}(\Omega )$
and $W^{k,q}(\Omega)$. Here, we use $\mathbb{N}$ for the set of natural numbers including zero.
For $s\geq 0$ and $q\in [1,+\infty )$, we denote by $H^{s,q}(\Omega )$ the Bessel-potential spaces and by $W^{s,q}(\Omega )$ the Slobodeckij spaces.
If $q=2$, it holds $H^{s,2}(\Omega)=W^{s,2}(\Omega )$ for all $s$ and these spaces are Hilbert spaces.
We use the notations $H^s(\Omega)=H^{s,2}(\Omega)=W^{s,2}(\Omega )$ and $H^0(\Omega)$ can be identified with $L^2(\Omega)$.
The Lebesgue spaces, Sobolev spaces and Slobodeckij spaces on the boundary $\Gamma$ can be defined analogously,
provided that $\Gamma$ is sufficiently regular.
Again, we write $H^s(\Gamma)=H^{s,2}(\Gamma)=W^{s,2}(\Gamma)$ and identify $H^0(\Gamma)$ with $L^2(\Gamma)$.
For convenience, we shall use the following shortcuts:
\begin{align*}
	&H:=L^{2}(\Omega),\quad H_{\Gamma}:=L^{2}(\Gamma),\quad V:=H^{1}(\Omega),\quad V_{\Gamma}:=H^{1}(\Gamma).
\end{align*}
Next, we introduce the product spaces
$$\mathcal{L}^{q}:=L^{q}(\Omega)\times L^{q}(\Gamma)\quad\mathrm{and}\quad\mathcal{H}^{k}:=H^{k}(\Omega)\times H^{k}(\Gamma),$$
for $q\in [1,+\infty]$ and $k\in \mathbb{R}$, $k\geq 0$.
Like before, we can identify  $\mathcal{H}^{0}$ with $\mathcal{L}^{2}$.
For any $k\in \mathbb{N}$, $\mathcal{H}^{k}$ is a Hilbert space endowed with the standard inner product
$$
\big((y,y_{\Gamma}),(z,z_{\Gamma})\big)_{\mathcal{H}^{k}}:=(y,z)_{H^{k}(\Omega)}+(y_{\Gamma},z_{\Gamma})_{H^{k}(\Gamma)},\quad\forall\, (y,y_{\Gamma}), (z,z_{\Gamma})\in\mathcal{H}^{k}
$$
and the induced norm $\Vert\cdot\Vert_{\mathcal{H}^{k}}:=(\cdot,\cdot)_{\mathcal{H}^{k}}^{1/2}$.
We introduce the duality pairing
\begin{align*}
	\big\langle (y,y_\Gamma),(\zeta, \zeta_\Gamma)\big\rangle_{(\mathcal{H}^1)',\mathcal{H}^1}
	= (y,\zeta)_{L^2(\Omega)}+ (y_\Gamma, \zeta_\Gamma)_{L^2(\Gamma)},
	\quad \forall\, (y,y_\Gamma)\in \mathcal{L}^2,\ (\zeta, \zeta_\Gamma)\in \mathcal{H}^1.
\end{align*}
By the Riesz representation theorem, this product
can be extended to a duality pairing on $(\mathcal{H}^1)'\times \mathcal{H}^1$.
For any positive integer $k$, we introduce the Hilbert space
$$
\mathcal{V}^{k}:=\big\{(y,y_{\Gamma})\in\mathcal{H}^{k}\;:\;y|_{\Gamma}=y_{\Gamma}\ \ \text{a.e. on }\Gamma\big\},
$$
endowed with the inner product $(\cdot,\cdot)_{\mathcal{V}^{k}}:=(\cdot,\cdot)_{\mathcal{H}^{k}}$ and the associated norm $\Vert\cdot\Vert_{\mathcal{V}^{k}}:=\Vert\cdot\Vert_{\mathcal{H}^{k}}$.
Here, $y|_{\Gamma}$ stands for the trace of $y\in H^k(\Omega)$ on the boundary $\Gamma$, which is meaningful for $k\in \mathbb{Z}^+$.
The duality pairing on $(\mathcal{V}^1)'\times \mathcal{V}^1$ can be defined in a similar manner.

For every $y\in V'$, we denote by $%
\langle y\rangle_\Omega=|\Omega|^{-1}\langle
y,1\rangle_{V',\,V}$ its generalized mean
value over $\Omega$. If $y\in L^1(\Omega)$, then its spatial mean is simply
given by $\langle y\rangle_\Omega=|\Omega|^{-1}\int_\Omega y \,\mathrm{d}x$.
The spatial mean for a function $y_\Gamma$ on $\Gamma$, denoted by $\langle
y_\Gamma\rangle_\Gamma$, can be defined in a similar manner. With these notations, we define the following subspaces for functions with zero mean:
\begin{align}
\begin{array}{ll}
H_{0}:=\big\{y\in H:\,\langle y\rangle_{\Omega}=0\big\},
&\qquad H_{\Gamma,0}:=\big\{y_{\Gamma}\in H_{\Gamma}:\,\langle y_{\Gamma}\rangle_{\Gamma}=0\big\},\notag\\
V_{0}:=V\cap H_{0},
&\qquad  V_{\Gamma,0}:=V_{\Gamma}\cap H_{\Gamma,0},\notag\\
V_{0}^\ast:=\big\{y\in V':\,\langle y\rangle_{\Omega}=0\big\},
&\qquad  V_{\Gamma,0}^\ast:=\big\{y_{\Gamma}\in V_{\Gamma}':\,\langle y_{\Gamma}\rangle_{\Gamma}=0\big\},
\end{array}
\end{align}
and
\begin{align}
\mathcal{L}^{2}_{(0)}:=H_{0}\times H_{\Gamma,0},\quad\mathcal{H}^{k}_{(0)}:=\mathcal{H}^{k}\cap\mathcal{L}^{2}_{(0)}, \quad\mathcal{V}^{k}_{(0)}:=\mathcal{V}^{k}\cap\mathcal{L}^{2}_{(0)}.\notag
\end{align}
We also define the projection operators
\begin{align}
&\mathbf{P}_{\Omega}:H\rightarrow H_{0},\quad y\mapsto y-\langle y\rangle_{\Omega},\quad\forall\,y\in H,\notag\\
&\mathbf{P}_{\Gamma}:H_{\Gamma}\rightarrow H_{\Gamma,0},\quad y_{\Gamma}\mapsto y_{\Gamma}-\langle y_{\Gamma}\rangle_{\Gamma},\quad\forall\,y_{\Gamma}\in H_{\Gamma}.\notag
\end{align}
Thanks to the Poincar\'e-Wirtinger inequality, there exists a positive constant $C_{p}$ such that
\begin{align}
&\Vert y\Vert_{V}^{2}\leq C_{p}\Vert y\Vert_{V_{0}}^{2},\quad\Vert y\Vert_{V_{0}}:=\Big(\int_{\Omega}|\nabla y|^{2}\,\mathrm{d}x\Big)^{\frac{1}{2}},\quad\forall\,y\in V_{0},
\label{bulkpoin}\\
&\Vert y_{\Gamma}\Vert_{V_{\Gamma}}^{2}\leq C_{p}\Vert y_{\Gamma}\Vert_{V_{\Gamma,0}}^{2},\quad\Vert y_{\Gamma}\Vert_{V_{\Gamma,0}}:=\Big(\int_{\Gamma}|\nabla_{\Gamma}y_{\Gamma}|^{2}\,\mathrm{d}S\Big)^{\frac{1}{2}}, \quad\forall\,y_{\Gamma}\in V_{\Gamma,0}.
\label{surfacepoin}
\end{align}
For $k\in\mathbb{N}$ and $K\in[0,+\infty)$, we introduce the following notations
\begin{align*}
	\mathcal{H}_{K}^{k}:=\left\{
	\begin{array}{ll}
		\mathcal{H}^{k}&\text{if }K>0,\\
		\mathcal{V}^{k}&\text{if }K=0,
	\end{array}
	\right.
	\qquad\text{and}\qquad
		\mathcal{H}_{K,0}^{k}:=\left\{
	\begin{array}{ll}
		\mathcal{H}_{(0)}^{k}&\text{if }K>0,\\
		\mathcal{V}_{(0)}^{k}&\text{if }K=0.
	\end{array}
	\right.
\end{align*}
Besides, we consider the Hilbert spaces
\begin{align*}
	\mathcal{W}_{K}^{2}:=\left\{
	\begin{array}{ll}
	\big\{(y,y_{\Gamma})\in\mathcal{H}^{2}:\,K\partial_{\mathbf{n}}y
	=y_{\Gamma}-y\text{ a.e. on }\Gamma\big\},&\text{if }K>0,\\
	\mathcal{V}^{2},&\text{if }K=0,
	\end{array}
	\right.\quad\text{and}\quad
	\mathcal{W}_{K,0}^{2}:=\mathcal{W}_{K}^{2}\cap\mathcal{L}_{(0)}^{2},
\end{align*}
equipped with the usual inner product $(\cdot,\cdot)_{\mathcal{H}^{2}}$ and the associated norm $\|\cdot\|_{\mathcal{H}^{2}}$.

Like in \cite{CFW}, from the Lax-Milgram theorem, we can introduce the operator $\mathcal{N}_{\Omega}:V_{0}^\ast\rightarrow V_{0}$ by $u=\mathcal{N}_{\Omega}v$ if and only if $\langle u\rangle_{\Omega}=0$ and
\begin{align*}
\int_{\Omega}\nabla u\cdot\nabla z\,\mathrm{d}x=\langle v,z\rangle_{V',V},\quad\forall\,z\in V.
\end{align*}
Analogously, we define $\mathcal{N}_{\Gamma}:V_{\Gamma,0}^\ast\rightarrow V_{\Gamma,0}$ by $u_{\Gamma}=\mathcal{N}_{\Gamma}v_{\Gamma}$ if and only if $\langle u_{\Gamma}\rangle_{\Gamma}=0$ and
\begin{align*}
\int_{\Gamma}\nabla_{\Gamma}u_{\Gamma}\cdot\nabla_{\Gamma}z_{\Gamma}\,\mathrm{d}S=\langle v_{\Gamma},z_{\Gamma}\rangle_{V_{\Gamma}',V_{\Gamma}},\quad\forall\,z_{\Gamma}\in V_{\Gamma}.
\end{align*}
By virtue of these definitions, we can introduce the following equivalent norms
\begin{align*}
&\Vert y\Vert_{V_{0}^*}:=\Big(\int_{\Omega}|\nabla\mathcal{N}_{\Omega}y|^{2}\,\mathrm{d}x\Big)^{1/2},
\qquad  \forall\, y\in V_{0}^*,\\
&\Vert y\Vert_{V'}:=\Big(\Vert y - \langle y\rangle_\Omega \Vert_{V_{0}^*}^2+ |\langle y\rangle_\Omega|^2\Big)^{1/2},
\qquad \forall\, y\in V',\\
&\Vert y_{\Gamma}\Vert_{V_{\Gamma,0}^*}:=\Big(\int_{\Gamma}|\nabla_{\Gamma}\mathcal{N}_{\Gamma}y_{\Gamma}|^{2}\,\mathrm{d}S\Big)^{1/2},
\qquad  \forall\,y_{\Gamma}\in V_{\Gamma,0}^*,\\
&\Vert y_{\Gamma}\Vert_{V_{\Gamma}'}:=\Big(\Vert y_{\Gamma}- \langle y_{\Gamma}\rangle_\Gamma \Vert_{V_{\Gamma,0}^*}^2 + |\langle y_{\Gamma}\rangle_\Gamma|^2\Big)^{1/2},
\qquad  \forall\,y_{\Gamma}\in V_{\Gamma}'.
\end{align*}
Finally, we recall the following chain rules:
\begin{align*}
	&\frac{1}{2}\frac{\mathrm{d}}{\mathrm{d}t}\Vert y\Vert_{V_{0}^\ast}^{2}=\langle \partial_t y,\mathcal{N}_{\Omega}y\rangle_{V',V},\quad\forall\,y\in H^{1}(0,T;V_{0}^\ast),\\
	&\frac{1}{2}\frac{\mathrm{d}}{\mathrm{d}t}\Vert y_{\Gamma}\Vert_{V_{\Gamma,0}^\ast}^{2}=\langle \partial_t y_{\Gamma},\mathcal{N}_{\Gamma}y_{\Gamma}\rangle_{V_{\Gamma}',V_{\Gamma}},
	\quad\forall\,y_{\Gamma}\in H^{1}(0,T;V_{\Gamma,0}^\ast).
\end{align*}

\subsection{The initial boundary value problem}
For an arbitrary but given final time $T\in(0,+\infty)$, we denote  $Q_{T}:=\Omega\times(0,T)$ and $\Sigma_{T}:=\Gamma\times(0,T)$. If $T=+\infty$, we simply set $Q:=\Omega\times(0,+\infty)$ and $\Sigma:=\Gamma\times(0,+\infty)$.
In view of the decomposition \eqref{decom} for the bulk and surface potentials,
we reformulate our target problem \eqref{CH}--\eqref{initial} as follows:
\begin{align}
\left\{
\begin{array}{ll}
\partial_{t}\varphi=\Delta\mu,&\text{in }Q,\\
\mu=-\Delta\varphi+\beta(\varphi)+\pi(\varphi),&\text{in }Q,\\
\partial_{\mathbf{n}}\mu=0,&\text{on }\Sigma,\\
K\partial_{\mathbf{n}}\varphi=\psi-\varphi,&\text{on }\Sigma,\\
\partial_{t}\psi=\Delta_{\Gamma}\theta,&\text{on }\Sigma,\\
\theta=\partial_{\mathbf{n}}\varphi-\Delta_{\Gamma}\psi+\beta_{\Gamma}(\psi)+\pi_{\Gamma}(\psi),&\text{on }\Sigma,\\
\varphi|_{t=0}=\varphi_{0},&\text{in }\Omega,\\
\psi|_{t=0}=\psi_{0},&\text{on }\Gamma,
\end{array}\right.\label{model}
\end{align}
where $\varphi_{0}:\Omega\rightarrow[-1,1]$, $\psi_{0}:\Gamma\rightarrow[-1,1]$ are given functions.

Throughout this paper, we make the following assumptions (cf. \cite{CFW,LvWu-2}):
\begin{description}
\item[$\mathbf{(A1)}$]
The nonlinear convex functions $\widehat{\beta}$, $\widehat{\beta}_{\Gamma}$ belong to $C([-1,1])\cap C^{2}(-1,1)$. Their derivatives are denoted by $\beta=\widehat{\beta}'$, $\beta_{\Gamma}=\widehat{\beta}_{\Gamma}'$ such that $\beta$, $\beta_{\Gamma}\in C^{1}(-1,1)$ are monotone increasing functions. Moreover, it holds
\begin{align*}
&\lim_{r\rightarrow-1}\beta(r)=-\infty,\quad\lim_{r\rightarrow-1}\beta_{\Gamma}(r)=-\infty,\\
&\lim_{r\rightarrow1}\beta(r)=+\infty,\quad\lim_{r\rightarrow1}\beta_{\Gamma}(r)=+\infty,
\end{align*}
and the derivatives $\beta'$, $\beta_{\Gamma}'$ fulfill
\begin{align*}
    \beta'(r)\geq\varpi,\quad\beta_{\Gamma}'(r)\geq\varpi,\quad\forall \,r\in (-1,1)
\end{align*}
for some constant $\varpi>0$. Without loss of generality, we set $\widehat{\beta}(0)=\widehat{\beta}_{\Gamma}(0)=\beta(0)=\beta_{\Gamma}(0)=0$ and make the extension $\widehat{\beta}(r)=+\infty$, $\widehat{\beta}_{\Gamma}(r)=+\infty$ for $|r|>1$.
\item[$\mathbf{(A2)}$]  There exist positive constants $\varrho$, $c_{0}$ such that
  \begin{align}
|\beta(r)|\leq \varrho|\beta_{\Gamma}(r)|+c_{0},\quad\forall\,r\in (-1,1).\label{assum1}
  \end{align}
\item[$\mathbf{(A3)}$] $\widehat{\pi}, \widehat{\pi}_\Gamma \in C^1(\mathbb{R})$ and their derivatives $\pi$, $\pi_\Gamma$ are globally Lipschitz continuous with Lipschitz constants denoted by $\gamma_{1}$ and $\gamma_{2}$, respectively.
\item[$\mathbf{(A4)}$] The initial datum satisfies $(\varphi_{0},\psi_{0})\in\mathcal{H}_{K}^{1}$, $\widehat{\beta}(\varphi_{0})\in L^{1}(\Omega)$, $\widehat{\beta}_{\Gamma}(\psi_{0})\in L^{1}(\Gamma)$. Moreover, it holds
    $$
    m_{0}:=\langle\varphi_{0}\rangle_{\Omega}\in(-1,1)
    \quad \text{and}\quad m_{\Gamma0}:=\langle\psi_{0}\rangle_{\Gamma}\in(-1,1).
    $$
\end{description}
\begin{remark}\rm
In order to handle the bulk-surface interaction,  it is necessary to have a compatibility condition between the bulk and surface potentials.
In $\mathbf{(A2)}$, we take a common assumption where $\beta_{\Gamma}$ dominates $\beta$ (as seen in  \cite{CC,CFW,CGNS,CGS14,CGS17,LW}). This choice allows us to derive some crucial uniform estimates without encountering further technical issues. An alternative approach is to consider the bulk potential as the dominating one (cf. \cite{GMS09,GMS}).
\end{remark}

For the sake of convenience, below we shall use the bold notations
$$
\bm{\varphi}=(\varphi, \psi), \quad \bm{\mu}=(\mu, \theta),\quad \bm{\beta}=(\beta,\beta_\Gamma),
\quad \bm{\pi}=(\pi, \pi_\Gamma), \quad \bm{\varphi}_0=(\varphi_0, \psi_0),
$$
and also for generic elements $\boldsymbol{y}=(y,y_\Gamma)$ in the product spaces $\mathcal{L}^{2}$, $\mathcal{H}^{1}$, $\mathcal{V}^{1}$ etc. As a preliminary, we have the following result on the well-posedness of problem \eqref{model}:
\begin{proposition}[Well-posedness]
\label{weakexist}
Suppose that $\Omega\subset\mathbb{R}^{d}$ $(d\in\{2,3\})$ is a bounded domain with smooth boundary $\Gamma$, $K\in[0,+\infty)$ and the assumptions $\mathbf{(A1)}$--$\mathbf{(A4)}$ are satisfied. For any $T>0$, problem \eqref{model} admits a unique global weak solution $(\boldsymbol{\varphi},\boldsymbol{\mu})$ on $[0,T]$ in the following sense:
\begin{align}
&\varphi\in H^{1}(0,T;V')\cap L^{\infty}(0,T;V)\cap L^{2}(0,T;H^{2}(\Omega)),\notag\\
&\mu\in L^{2}(0,T;V),\quad\beta(\varphi)\in L^{2}(0,T;H),\notag\\
&\psi\in H^{1}(0,T;V_{\Gamma}')\cap L^{\infty}(0,T;V_{\Gamma})\cap L^{2}(0,T;H^{2}(\Gamma)),\notag\\
&\theta\in L^{2}(0,T;V_{\Gamma}),\quad\beta_{\Gamma}(\psi)\in L^{2}(0,T;H_{\Gamma}),\notag
\end{align}
such that
\begin{align}
&\langle\partial_{t}\varphi,z\rangle_{V',V}+\int_{\Omega}\nabla\mu\cdot\nabla z\,\mathrm{d}x
=0,\qquad\forall\,z\in V,\label{eq2.7}\\
&\langle\partial_{t}\psi,z_{\Gamma}\rangle_{V_{\Gamma}',V_{\Gamma}}
+\int_{\Gamma}\nabla_{\Gamma}\theta\cdot\nabla_{\Gamma}z_{\Gamma}\,\mathrm{d}S=0,
\qquad\forall\,z_{\Gamma}\in V_{\Gamma},\label{eq2.10}
\end{align}
for almost all $t\in(0,T)$ and
\begin{align}
&\mu=-\Delta\varphi+\beta(\varphi)+\pi(\varphi),&&\text{a.e. in }Q_{T},\label{eq2.8}\\
&K\partial_{\mathbf{n}}\varphi=\psi-\varphi,&&\text{a.e. on }\Sigma_{T},\label{eq2.9}\\
&\theta=\partial_{\mathbf{n}}\varphi-\Delta_{\Gamma}\psi+\beta_{\Gamma}(\psi)+\pi_{\Gamma}(\psi),&&\text{a.e. on }\Sigma_{T}.\label{eq2.11}
\end{align}
Moreover, the initial conditions are satisfied
\begin{align}
\varphi|_{t=0}=\varphi_{0}\ \text{ a.e. in }\Omega,\qquad\psi|_{t=0}=\psi_{0}\ \text{ a.e. on }\Gamma.\notag
\end{align}
Let $(\boldsymbol{\varphi}^{(i)},\boldsymbol{\mu}^{(i)})$ be two weak solutions to problem \eqref{model} corresponding to two given initial data $\boldsymbol{\varphi}_{0}^{(i)}$ $(i\in\{1,2\})$ that satisfy $\mathbf{(A4)}$. Then, there exists a positive constant $C$, depending on $\gamma_{1}$, $\gamma_{2}$, $\Omega$, $\Gamma$, $T$ and coefficients of system, such that
\begin{align}
	&\Vert\varphi^{(1)}-\varphi^{(2)}\Vert_{C([0,T];V')} +\Vert\psi^{(1)}-\psi^{(2)}\Vert_{C([0,T];V_{\Gamma}')} +\Vert\varphi^{(1)}-\varphi^{(2)}\Vert_{L^{2}(0,T;V)} +\Vert\psi^{(1)}-\psi^{(2)}\Vert_{L^{2}(0,T;V_{\Gamma})}
\notag\\
	&\quad\leq C\Big(\Vert\varphi_{0}^{(1)}-\varphi_{0}^{(2)}\Vert_{V'}
	+\Vert\psi_{0}^{(1)}-\psi_{0}^{(2)}\Vert_{V_{\Gamma}'}\Big).  \label{contidepen}
\end{align}
\end{proposition}

\begin{remark}\rm
	\label{V-conti}
The uniqueness of weak solutions is a straightforward consequence of the continuous dependence estimate \eqref{contidepen}. Since $T>0$ is arbitrary, the solution is well-defined on the interval $[0,+\infty)$. Additionally, due to the Sobolev embedding theorem and the Aubin-Lions-Simon lemma (see Lemma \ref{ALS}), the regularity of the bulk and surface phase functions leads to the continuity property
$\boldsymbol{\varphi}\in C_{w}([0,+\infty);\mathcal{H}_{K}^{1})\cap C([0,+\infty);\mathcal{L}^{2})$. This ensures that the initial data can be attained.
\end{remark}

\subsection{Statement of main results}
We are now ready to present our main results. When discussing the strict separation property and long-time behavior of global weak solutions, we examine the problem over the entire time interval $[0,+\infty)$. For the double obstacle limit, we shall work on a finite time interval $[0,T]$ for an arbitrary but fixed final time $T\in(0,+\infty)$.

\begin{theorem}[Eventual separation]
\label{eventual}
Suppose that $\Omega\subset\mathbb{R}^{d}$ $(d\in\{2,3\})$ is a bounded domain with smooth boundary $\Gamma$, $K\in[0,+\infty)$ and $\mathbf{(A1)}$--$\mathbf{(A4)}$ are satisfied. Let $(\boldsymbol{\varphi},\boldsymbol{\mu})$ be the unique global weak solution to problem \eqref{model} obtained in Proposition \ref{weakexist}. There exist a constant  $\delta_{1}\in(0,1)$ and a sufficiently large time $T_{\mathrm{SP}}\gg 1$ such that
\begin{align}
\Vert\varphi(t)\Vert_{L^{\infty}(\Omega)}\leq1-\delta_{1},\quad\Vert\psi(t)\Vert_{L^{\infty}(\Gamma)}\leq1-\delta_{1},\quad\forall \;t\geq T_{\mathrm{SP}}.\label{2.1}
\end{align}
\end{theorem}

If one of the following additional assumptions is fulfilled, we can establish the instantaneous strict separation property in two dimensions (cf. \cite{GGG,GP}):
\begin{description}
	\item[$\mathbf{(A5a)}$] There exist constants $C_{\sharp}>0$ and $\gamma_{\sharp}\in[1,2)$ such that
	\begin{align}
		\beta'(r)\leq C_{\sharp}e^{C_{\sharp}|\beta(r)|^{\gamma_{\sharp}}},\quad\forall\,r\in(-1,1).\notag%\label{additional}
	\end{align}
	\item[$\mathbf{(A5b)}$] As $\delta\rightarrow0^{+}$, for some $\kappa>1/2$, it holds
	\begin{align}
		\frac{1}{\beta(1-2\delta)}=O\Big(\frac{1}{|\ln \delta|^{\kappa}}\Big),\quad	\frac{1}{|\beta(-1+2\delta)|}=O\Big(\frac{1}{|\ln \delta|^{\kappa}}\Big).\notag
	\end{align}
\end{description}
\begin{theorem}[Instantaneous separation in 2D]
\label{intantaneous}
Suppose that $\Omega\subset\mathbb{R}^{2}$ is a bounded domain with smooth boundary $\Gamma$, $K\in[0,+\infty)$ and $\mathbf{(A1)}$--$\mathbf{(A4)}$ are satisfied. Assume in addition,  $\mathbf{(A5a)}$ or $\mathbf{(A5b)}$ holds. Let $(\boldsymbol{\varphi},\boldsymbol{\mu})$ be the unique global weak solution to problem \eqref{model} obtained in Proposition \ref{weakexist}. Then for any $\tau>0$, there exists a constant $\delta_{2}\in(0,1)$, depending on $\tau$, $m_{0}$, $m_{\Gamma0}$, $E(\boldsymbol{\varphi}_{0})$, $\Omega$, $\Gamma$ and coefficients of system, such that
\begin{align}
\Vert \varphi(t)\Vert_{L^{\infty}(\Omega)}\leq1-\delta_{2},\quad\Vert \psi(t)\Vert_{L^{\infty}(\Gamma)}\leq1-\delta_{2}, \quad\forall\,t\geq\tau.\notag
\end{align}
\end{theorem}

%\subsubsection{Convergence to equilibrium}

Thanks to the eventual separation property obtained in Theorem \ref{eventual}, we are able to show that every global weak solution  converges to a single equilibrium as $t\to +\infty$.
\begin{theorem}[Convergence to equilibrium]
\label{equilibrium}
Suppose that the assumptions in Theorem \ref{eventual} are satisfied.
Assume in addition, $\beta$, $\beta_{\Gamma}$ are real analytic on $(-1,1)$ and $\pi$, $\pi_{\Gamma}$ are real analytic on $\mathbb{R}$. Let $\boldsymbol{\varphi}$ be the unique global weak solution to problem  \eqref{model} obtained in Proposition \ref{weakexist}. We have
\begin{align*}
\lim_{t\rightarrow+\infty} \Vert\boldsymbol{\varphi}(t)-\boldsymbol{\varphi}_{\infty}\Vert_{\mathcal{H}^{2}}=0,
\end{align*}
where $\boldsymbol{\varphi}_{\infty}:=(\varphi_{\infty},\psi_{\infty})$ is a steady state that satisfies the following elliptic problem
\begin{align}
\left\{
\begin{array}{ll}
-\Delta \varphi_{\infty}+\beta(\varphi_{\infty})+\pi(\varphi_{\infty})=\mu_{\infty},&\text{in }\Omega,\\
-\Delta_{\Gamma}\psi_{\infty}+\beta_{\Gamma}(\psi_{\infty})+\pi_{\Gamma}(\psi_{\infty})
+\partial_{\mathbf{n}}\varphi_{\infty}=\theta_{\infty},&\text{on }\Gamma,\\
K\partial_{\mathbf{n}}\varphi_{\infty}=\psi_{\infty}-\varphi_{\infty},&\text{on }\Gamma,
\end{array}\right.\notag
\end{align}
with $\langle\varphi_{\infty}\rangle_{\Omega}=m_{0}$, $\langle\psi_{\infty}\rangle_{\Gamma}=m_{\Gamma0}$ and
\begin{align*}
	&\mu_{\infty}=\frac{1}{|\Omega|}\Big[\int_{\Omega} \big(\beta(\varphi_{\infty})+\pi(\varphi_{\infty})\big)\,\mathrm{d}x
	-\int_{\Gamma}\partial_{\mathbf{n}}\varphi_{\infty}\,\mathrm{d}S\Big],\\
	&\theta_{\infty}=\frac{1}{|\Gamma|}\int_{\Gamma} \Big(\beta_{\Gamma}(\psi_{\infty})+\pi_{\Gamma}(\psi_{\infty}) +\partial_{\mathbf{n}}\varphi_{\infty}\Big)\,\mathrm{d}S.
\end{align*}
Moreover, the following estimate on convergence rate holds
\begin{align}
\Vert\boldsymbol{\varphi}(t)-\boldsymbol{\varphi}_{\infty}\Vert_{\mathcal{H}^{1}}\leq C\big(1+t\big)^{-\frac{\varsigma^{*}}{1-2\varsigma^{*}}},
\quad\forall\,t\geq0,\label{conver-rate}
\end{align}
where $\varsigma^{*}\in(0,1/2)$ is a constant depending on $\boldsymbol{\varphi}_{\infty}$, $m_{0}$, $m_{\Gamma0}$, $\Omega$, $\Gamma$ and coefficients of system, the positive constant $C$ may further depend on $E(\boldsymbol{\varphi}_{0})$.
\end{theorem}

%\subsubsection{Double obstacle limit}
The last result is about the double obstacle limit. To this end, recalling \eqref{logarithmic}, for every $\Theta\in(0,1]$, we take the bulk and boundary free energy densities $F,G$ in problem \eqref{model} as
$$
 F_{\Theta}(\varphi)=\frac{\Theta}{2} F_0(\varphi)-\frac{\Theta_c}{2}\varphi^{2},\quad
 G_{\Theta}(\psi)=\frac{\Theta}{2} F_0(\psi)-\frac{\Theta_c}{2}\psi^{2},\quad\text{with }F_{0}\text{ given in }\eqref{logarithmic}.
 %\quad \text{where}\   f_0(r)=\ln\left(\frac{1+r}{1-r}\right),
$$
It is easy to check that the assumptions $\mathbf{(A1)}$--$\mathbf{(A3)}$ are fulfilled. Besides, for all $\boldsymbol{\varphi}\in \mathcal{H}_K^1$, $\widehat{\beta}(\varphi)\in L^{1}(\Omega)$, $\widehat{\beta}_{\Gamma}(\psi)\in L^{1}(\Gamma)$, we find
\begin{align*}
\lim_{\Theta\rightarrow0}E_{\Theta}(\boldsymbol{\varphi})
&=\frac{1}{2}\int_{\Omega}|\nabla \varphi|^{2}\,\mathrm{d}x +\int_{\Omega}I_{[-1,1]}(\varphi)\,\mathrm{d}x -\frac{\Theta_{c}}{2}\int_\Omega \varphi^{2}\,\mathrm{d}x +\frac{1}{2}\int_{\Gamma}|\nabla_{\Gamma}\psi|^{2}\,\mathrm{d}S\\
&\quad
+\int_{\Gamma}I_{[-1,1]}(\psi)\,\mathrm{d}S
- \frac{\Theta_{c}}{2}\int_\Gamma \psi^{2}\,\mathrm{d}S +\frac{\chi(K)}{2}\int_{\Gamma}|\psi-\varphi|^{2}\,\mathrm{d}S\\
&=:E_{0}(\boldsymbol{\varphi}).
\end{align*}
Hence, we denote the problem \eqref{model} correspond to $\Theta\in (0,1]$ by $(S_{\Theta})$ with the associated total free energy $E_{\Theta}(\boldsymbol{\varphi})$, and denote the limit problem related to $E_{0}$ by $(S_{0})$, that is,
\begin{align}
 (S_{0}) \
\left\{
\begin{array}{ll}
\partial_{t}\varphi=\Delta\mu,&\text{in }Q,\\
\mu+\Delta\varphi+\Theta_{c}\varphi\in\partial I_{[-1,1]}(\varphi),&\text{in }Q,\\
\partial_{\mathbf{n}}\mu=0,&\text{on }\Sigma,\\
K\partial_{\mathbf{n}}\varphi=\psi-\varphi,&\text{on }\Sigma,\\
\partial_{t}\psi=\Delta_{\Gamma}\theta,&\text{on }\Sigma,\\
\theta -\partial_{\mathbf{n}}\varphi+ \Delta_{\Gamma}\psi+ \Theta_{c}\psi \in\partial I_{[-1,1]}(\psi),&\text{on }\Sigma,\\
\varphi|_{t=0}=\varphi_{0},&\text{in }\Omega,\\
\psi|_{t=0}=\psi_{0},&\text{on }\Gamma.
\end{array}\right.\label{model-2obs}
\end{align}

The following result implies that weak solutions of problem $(S_{\Theta})$ converge to the weak solution of problem $(S_{0})$ as $\Theta\to 0$.
\begin{theorem}[Double obstacle limit]
	\label{doubleobstacle}
Suppose that $\Omega\subset\mathbb{R}^{d}$ $(d\in\{2,3\})$ is a bounded domain with smooth boundary $\Gamma$, $K\in [0,+\infty)$ and $T>0$ is an arbitrary but fixed final time.
Let $0<\Theta_k\leq 1$, $k\in \mathbb{Z}^+$ be such that $\lim_{k\to +\infty}\Theta_k=0$.
Moreover, we assume $\boldsymbol{\varphi}_{0,k}$, $\boldsymbol{\varphi}_{0}\in \mathcal{H}_K^1$ satisfying $\lim_{k\to +\infty}\|\boldsymbol{\varphi}_{0,k}- \boldsymbol{\varphi}_{0}\|_{\mathcal{H}^1}=0$ and
\begin{align*}
& \|\varphi_{0,k}\|_{L^\infty(\Omega)}\leq 1,
\quad \|\psi_{0,k}\|_{L^\infty(\Gamma)}\leq 1,
\quad \sup_{k\in \mathbb{Z}^+}|\langle\varphi_{0,k}\rangle_\Omega|<1,
\quad \sup_{k\in \mathbb{Z}^+}|\langle\psi_{0,k}\rangle_\Gamma|<1.
\end{align*}
Let $(\boldsymbol{\varphi}_{\Theta_k},\boldsymbol{\mu}_{\Theta_k})$ be the unique weak solution to problem $(S_{\Theta_k})$ with the initial data $\boldsymbol{\varphi}_{0,k}$ and the nonlinearities
$F_{\Theta_k}(\varphi_{\Theta_k})$, $G_{\Theta_k}(\psi_{\Theta_k})$.
Then there exist limit functions $(\widetilde{\boldsymbol{\varphi}},\widetilde{\boldsymbol{\mu}},\widetilde{\boldsymbol{\xi}})$
such that as $k\to +\infty$, it holds
		\begin{align}
		\varphi_{\Theta_k}&\rightarrow\widetilde{\varphi}&&\text{weakly in }H^{1}(0,T;V')\cap L^{2}(0,T;H^{2}(\Omega))\notag\\
		&  &&\text{weakly star in }L^{\infty}(0,T;V)\notag\\
		&  &&\text{strongly in }C([0,T];H)\label{conver1}\\
        \psi_{\Theta_k}&\rightarrow\widetilde{\psi}&&\text{weakly in }H^{1}(0,T;V_{\Gamma}')\cap L^{2}(0,T;H^{2}(\Gamma))\notag\\
		&  &&\text{weakly star in }L^{\infty}(0,T;V_{\Gamma})\notag\\
		&  &&\text{strongly in }C([0,T];H_{\Gamma})\label{conver4}\\
		\mu_{\Theta_k}&\rightarrow\widetilde{\mu}&&\text{weakly in }L^{2}(0,T;V)\label{conver2}\\		
		\theta_{\Theta_k}&\rightarrow\widetilde{\theta}&&\text{weakly in }L^{2}(0,T;V_{\Gamma})\label{conver5}\\
\Theta_k f_{0}(\varphi_{\Theta_k})&\rightarrow\widetilde{\xi}&&\text{weakly in }L^{2}(0,T;H)\label{conver3}\\
		\Theta_k f_{0}(\psi_{\Theta_k})&\rightarrow\widetilde{\xi}_{\Gamma}&&\text{weakly in }L^{2}(0,T;H_{\Gamma})\label{conver6}
	\end{align}
where $ \widetilde{\mu}+\Delta\widetilde{\varphi}+\Theta_{c}\widetilde{\varphi} = \widetilde{\xi} \in\partial I_{[-1,1]}(\widetilde{\varphi})$ almost everywhere in $Q_T$  and $\widetilde{\theta}-\partial_\mathbf{n}\widetilde{\varphi} +\Delta_\Gamma\widetilde{\psi}+\Theta_{c}\widetilde{\psi} = \widetilde{\xi}_\Gamma \in\partial I_{[-1,1]}(\widetilde{\psi})$ almost everywhere in $\Sigma_T$. Furthermore, $(\boldsymbol{\widetilde{\varphi}},\boldsymbol{\widetilde{\mu}},\boldsymbol{\widetilde{\xi}})$ is a weak solution to problem \eqref{model-2obs} on $[0,T]$.
\end{theorem}

\section{Basic Properties of Global Weak Solutions}
\label{globalregularity}
\setcounter{equation}{0}

In this section, we present some preliminary results for problem \eqref{model} with $K\in [0,+\infty)$.
This includes examining the well-posedness, mass conservation, energy dissipation and regularity propagation of global weak solutions.

\subsection{Well-posedness}
For completeness and convenience of the subsequent analysis, we sketch the proof of Proposition \ref{weakexist}. To this end, let us recall the approximating problem considered in \cite{CFW,CKSS}.  For every $\varepsilon\in (0,1)$, set $\beta_{\varepsilon}$, $\beta_{\Gamma,\varepsilon}:\mathbb{R}\rightarrow\mathbb{R}$, along with the associated resolvent operators $J_{\varepsilon}$, $J_{\Gamma,\varepsilon}:\mathbb{R}\rightarrow\mathbb{R}$ given by (see also \cite{CC,CF15})
\begin{align*} &\beta_{\varepsilon}(r):=\frac{1}{\varepsilon}\big(r-J_{\varepsilon}(r)\big) :=\frac{1}{\varepsilon}\big(r-(I+\varepsilon\beta)^{-1}(r)\big), \\ &\beta_{\Gamma,\varepsilon}(r):=\frac{1}{\varepsilon\varrho} \big(r-J_{\Gamma,\varepsilon}(r)\big) :=\frac{1}{\varepsilon\varrho}\big(r-(I+\varepsilon\varrho\beta_{\Gamma})^{-1}(r)\big),
\end{align*}
for all $r\in\mathbb{R}$, where $\varrho>0$ is the same constant as in the condition \eqref{assum1}.
Then the related Moreau-Yosida regularizations $\widehat{\beta}_{\varepsilon}$, $\widehat{\beta}_{\Gamma,\varepsilon}$ of $\widehat{\beta}$, $\widehat{\beta}_{\Gamma}:\mathbb{R}\rightarrow\mathbb{R}$ are then given by (see, e.g., \cite{R.E.S})
\begin{align}
	&\widehat{\beta}_{\varepsilon}(r) :=\inf_{s\in\mathbb{R}}\left\{\frac{1}{2\varepsilon}|r-s|^{2} +\widehat{\beta}(s)\right\}=\frac{1}{2\varepsilon}|r-J_{\varepsilon}(r)|^{2} +\widehat{\beta}\big(J_{\varepsilon}(r)\big) =\int_{0}^{r}\beta_{\varepsilon}(s)\,\mathrm{d}s,
	\notag\\
	&\widehat{\beta}_{\Gamma,\varepsilon}(r) :=\inf_{s\in\mathbb{R}}\left\{\frac{1}{2\varepsilon\varrho}|r-s|^{2} +\widehat{\beta}_{\Gamma}(s)\right\}=\int_{0}^{r}\beta_{\Gamma,\varepsilon}(s)\,\mathrm{d}s.
	\notag
\end{align}
Then for any $\varepsilon,\sigma\in(0,1)$, we consider the following approximating problem:
\begin{align}
\left\{
	\begin{array}{ll}
		\partial_{t}\varphi_{\varepsilon,\sigma}=\Delta\mu_{\varepsilon,\sigma}, &\quad\text{a.e. in }Q_{T},\\
		\mu_{\varepsilon,\sigma}=\sigma\partial_{t}\varphi_{\varepsilon,\sigma} -\Delta\varphi_{\varepsilon,\sigma}
		+\beta_{\varepsilon}(\varphi_{\varepsilon,\sigma}) +\pi(\varphi_{\varepsilon,\sigma}),&\quad\text{a.e. in }Q_{T},\\
		\partial_{\mathbf{n}}\mu_{\varepsilon,\sigma}=0,
&\quad\text{a.e. on }\Sigma_{T},\\
		K\partial_{\mathbf{n}}\varphi_{\varepsilon,\sigma} =\psi_{\varepsilon,\sigma}-\varphi_{\varepsilon,\sigma},
&\quad\text{a.e. on }\Sigma_{T},\\		\partial_{t}\psi_{\varepsilon,\sigma}=\Delta_{\Gamma}\theta_{\varepsilon,\sigma}, &\quad\text{a.e. on }\Sigma_{T},\\
		\theta_{\varepsilon,\sigma}=\sigma\partial_{t}\psi_{\varepsilon,\sigma} +\partial_{\mathbf{n}}\varphi_{\varepsilon,\sigma}
		-\Delta_{\Gamma}\psi_{\varepsilon,\sigma}
+\beta_{\Gamma,\varepsilon}(\psi_{\varepsilon,\sigma})
		+\pi_{\Gamma}(\psi_{\varepsilon,\sigma}),
&\quad\text{a.e. on }\Sigma_{T},\\
		\varphi_{\varepsilon,\sigma}|_{t=0}=\varphi_{0},
&\quad\text{a.e. in }\Omega,\\
		\psi_{\varepsilon,\sigma}|_{t=0}=\psi_{0},
&\quad\text{a.e. on }\Gamma.
	\end{array}
\right.\label{appro1}
\end{align}
\begin{remark}{\rm
In \eqref{appro1}, we have included two viscous terms $\sigma\partial_{t}\varphi_{\varepsilon,\sigma}$ and $\sigma\partial_{t}\psi_{\varepsilon,\sigma}$ in the bulk and surface chemical potentials. This allows us to achieve better regularity for the time derivative $\partial_{t}\boldsymbol{\varphi}_{\varepsilon,\sigma}$, and to handle the cases $K=0$, $K\in(0,+\infty)$ in a unified manner. As demonstrated in \cite{CKSS}, these two viscous terms are not necessary for the existence of global weak solutions to the approximating problem when $K\in (0,+\infty)$.}
\end{remark}

% will be provided in Section \ref{app-sys}.  and the continuous dependence estimate follows from the  energy method (see \cite{CFW} for the case $K=0$).

\begin{proof}[\textbf{Sketch of the Proof for Proposition \ref{weakexist}}.]
The existence of weak solutions can be found in \cite{CFW,CKSS}. Since the parameter $K$ yields different structure for the system, different methods are required. For the case $K=0$, the proof relies on a suitable  time-discretization scheme combined with the general theory of the maximal monotone operator and the compactness argument (see \cite{CFW}). For the case $K\in(0,+\infty)$, the solution can be constructed directly using a Faedo-Galerkin scheme and the compactness argument (see \cite{CKSS}). By considering the asymptotic limit $K\to 0$, this also provides an alternative approach to obtain the existence of weak solutions with $K=0$. Anyway, for $K\in [0,+\infty)$, we can conclude that problem \eqref{appro1} admits a unique weak solution satisfying
\begin{align*}
&\boldsymbol{\varphi}_{\varepsilon,\sigma}\in H^{1}(0,T;\mathcal{L}^{2})\cap L^{\infty}(0,T;\mathcal{H}_{K}^{1})\cap L^{2}(0,T;\mathcal{H}_{K}^{2}),\\
&\boldsymbol{\mu}_{\varepsilon,\sigma}\in L^{2}(0,T;\mathcal{H}^{1}).
\end{align*}
In this manner, we can obtain a family of approximating solutions $(\boldsymbol{\varphi}_{\varepsilon,\sigma},\boldsymbol{\mu}_{\varepsilon,\sigma})$ satisfying sufficient \emph{a priori} estimates that are uniform with respect to the approximating parameters $\varepsilon,\sigma\in(0,1)$. By the compactness argument and passing to the limit as $(\varepsilon,\sigma)\to(0,0)$ (up to a subsequence), we can find a limit pair $(\boldsymbol{\varphi},\boldsymbol{\mu})$ that gives the global weak solution to the original problem \eqref{model} on $[0,T]$.
The continuous dependence estimate \eqref{contidepen} in the case with $K=0$ has been proven in \cite[Theorem 2.4]{CFW} using the energy method. The argument therein can be extended to the case $K\in (0,+\infty)$ with a minor modification. The details are omitted here.
\end{proof}

\subsection{Mass conservation and energy equality}
\begin{proposition}[Mass conservation]
For all $t\geq0$, it holds
\begin{align}
  \langle\varphi(t)\rangle_{\Omega}=m_{0},\quad \langle\psi(t)\rangle_{\Gamma}=m_{\Gamma0}.\label{massconser}
\end{align}
\end{proposition}
\begin{proof}
	Taking $z=1$, $z_{\Gamma}=1$ in the weak formulations \eqref{eq2.7} and \eqref{eq2.10}, respectively, we easily arrive at the conclusion \eqref{massconser}.
\end{proof}

\begin{lemma}
	\label{ener}
Let $(\boldsymbol{\varphi},\boldsymbol{\mu})$ be the global weak solution obtained in Proposition \ref{weakexist}. It holds
\begin{align}
	E\big(\boldsymbol{\varphi}(t)\big)+\int_{0}^{t}\Big(\|\partial_t\varphi(s)\|_{V_{0}^\ast}^{2}
	+\|\partial_t\psi(s)\|_{V_{\Gamma,0}^\ast}^{2}\Big)\,\mathrm{d}s
	\leq E(\boldsymbol{\vp}_{0}),\quad\text{for a.a. }t\geq0.\label{energyin}
\end{align}
Moreover, there exists a positive constant $M_{1}$ such that
\begin{align}
&\|\boldsymbol{\varphi}\|_{L^{\infty}(0,+\infty;\mathcal{H}_{K}^{1})}
+\int_{0}^{+\infty}\Big(\|\partial_t\varphi(t)\|_{V_{0}^{\ast}}^{2}
+\|\partial_t\psi(t)\|_{V_{\Gamma,0}^{\ast}}^{2}\Big)\,\mathrm{d}t
\leq M_{1},\label{eq3.7}\\
&\int_{0}^{+\infty}\Big(\|\nabla\mu(t)\|_{H}^{2}
+\|\nabla_{\Gamma}\theta(t)\|_{H_{\Gamma}}^{2}\Big)\,\mathrm{d}t\leq M_{1}.\label{eq3.7'}
\end{align}
\end{lemma}
\begin{proof}
We first consider the approximating system \eqref{appro1} with $\varepsilon=\sigma$. Testing \eqref{appro1}$_{2}$, \eqref{appro1}$_{6}$ by $\partial_t\varphi_{\varepsilon,\sigma}\in L^{2}(0,+\infty;H)$ and $\partial_t\psi_{\varepsilon,\sigma}\in L^{2}(0,+\infty;H_{\Gamma})$, using the chain rule in \cite[Proposition A.1]{CKSS}, we find that $t\mapsto E_{\varepsilon}\big(\boldsymbol{\varphi}_{\varepsilon,\sigma}(t)\big)$ is absolutely continuous on $[0,+\infty)$ and
\begin{align}
	\frac{\mathrm{d}}{\mathrm{d}t}E_{\varepsilon}\big(\boldsymbol{\varphi}_{\varepsilon,\sigma}(t)\big)
	+ \sigma\Vert\partial_t\boldsymbol{\varphi}_{\varepsilon,\sigma}(t)\Vert_{\mathcal{L}^{2}}^{2} +\Vert\partial_t\varphi_{\varepsilon,\sigma}(t)\Vert_{V_{0}^\ast}^{2}
	+\|\partial_t\psi_{\varepsilon,\sigma}(t)\|_{V_{\Gamma,0}^\ast}^{2} =0, \quad\text{for a.a. }  t>0,
	\label{L1}
\end{align}
where
\begin{align}
	E_{\varepsilon}(\boldsymbol{y})&=\frac{1}{2}\int_{\Omega}|\nabla y|^{2}\,\mathrm{d}x +\frac{1}{2}\int_{\Gamma}|\nabla_{\Gamma}y_{\Gamma}|^{2}\,\mathrm{d}S
	+\frac{\chi(K)}{2}\int_{\Gamma}|y_{\Gamma}-y|^{2}\,\mathrm{d}S\notag\\ &\quad+\int_{\Omega}\widehat{\beta}_{\varepsilon}(y)+\widehat{\pi}(y)\,\mathrm{d}x
	+\int_{\Gamma}\widehat{\beta}_{\Gamma,\varepsilon}(y_{\Gamma}) +\widehat{\pi}_{\Gamma}(y_{\Gamma})\,\mathrm{d}S,
	\quad\forall\,\boldsymbol{y}\in \mathcal{H}_{K}^{1}.
	\notag
\end{align}
Integrating \eqref{L1} with respective to time, we find for all $t\geq 0$,
\begin{align}
	E_{\varepsilon}\big(\boldsymbol{\varphi}_{\varepsilon,\sigma}(t)\big) +
	\sigma\int_{0}^{t}\Vert\partial_t\boldsymbol{\varphi}_{\varepsilon,\sigma}(s)\Vert_{\mathcal{L}^{2}}^{2}\,\mathrm{d}s +\int_{0}^{t}\Big(\Vert\partial_t\varphi_{\varepsilon,\sigma}(s)\Vert_{V_{0}^*}^{2}
	+\|\partial_t\psi_{\varepsilon,\sigma}(s)\|_{V_{\Gamma,0}^\ast}^{2}\Big)\,\mathrm{d}s  =E_{\varepsilon}(\boldsymbol{\varphi}_{0}).
\label{eq4.1}
\end{align}
Recalling the weak and strong convergence results as $\varepsilon=\sigma\to0$ for any $T>0$ (in the sense of subsequence, cf. \cite{CFW,CKSS})
\begin{align*}
	\boldsymbol{\varphi}_{\varepsilon,\sigma}\rightarrow\boldsymbol{\varphi}&\qquad\text{weakly star in }L^{\infty}(0,T;\mathcal{H}_{K}^{1}),\\
	\boldsymbol{\varphi}_{\varepsilon,\sigma}\rightarrow \boldsymbol{\varphi}&\qquad\text{weakly in } L^{2}(0,T;\mathcal{H}_{K}^{2}),\\	\partial_t\varphi_{\varepsilon,\sigma}\to\partial_t\varphi&\qquad\text{weakly in }L^{2}(0,T;V_{0}^{\ast}),\\
	\partial_t\psi_{\varepsilon,\sigma}\to\partial_t\psi&\qquad\text{weakly in }L^{2}(0,T;V_{\Gamma,0}^{\ast})\\
	\boldsymbol{\varphi}_{\varepsilon,\sigma}\rightarrow \boldsymbol{\varphi}&\qquad\text{strongly in }C([0,T];\mathcal{L}^{2})\cap L^{2}(0,T;\mathcal{H}_{K}^{1}),\\
	\sigma\boldsymbol{\varphi}_{\varepsilon,\sigma}\rightarrow\boldsymbol{0} &\qquad\text{strongly in }H^{1}(0,T;\mathcal{L}^{2}),
\end{align*}
and by the lower weak semicontinuity of norms, we get
\begin{align*}
	\liminf_{\varepsilon\rightarrow0}\,\int_{0}^{t}\Big(\Vert\partial_t\varphi_{\varepsilon,\sigma}(s) \Vert_{V_{0}^\ast}^{2}
	+\|\partial_t\psi_{\varepsilon,\sigma}(s)\|_{V_{\Gamma,0}^\ast}^{2}\Big)\,\mathrm{d}s \geq\int_{0}^{t}\Big(\Vert\partial_t\varphi(s)\Vert_{V_{0}^\ast}^{2} +\|\partial_t\psi(s)\|_{V_{\Gamma,0}^\ast}^{2}\Big)\,\mathrm{d}s, \end{align*}
 for almost all $t\geq0$.
Besides, it is straightforward to check that
\begin{align*}
	&\liminf_{\varepsilon\rightarrow0}E_{\varepsilon}\big(\boldsymbol{\varphi}_{\varepsilon,\sigma}(t)\big) \geq  E\big(\boldsymbol{\varphi}(t)\big)\quad\text{for a.a. }t\geq0
	\quad\text{and}\quad
	\lim_{\varepsilon\rightarrow0} E_{\varepsilon}(\boldsymbol{\varphi}_0)= E(\boldsymbol{\varphi}_0).
\end{align*}
Taking lim inf as $\varepsilon=\sigma\rightarrow0$ in \eqref{eq4.1}, we arrive at the energy
inequality \eqref{energyin}.

Thanks to $\mathbf{(A1)}$, $\mathbf{(A3)}$, we can find a nonnegative constant $c_{1}$ such that
\begin{align}
	\widehat{\beta}(r)+\widehat{\pi}(r)\geq-c_{1}, \quad\widehat{\beta}_{\Gamma}(r)+\widehat{\pi}_{\Gamma}(r)\geq-c_{1}, \quad\forall\,r\in[-1,1].\notag
\end{align}
Hence, by the definition of $E(\boldsymbol{\varphi})$, it holds
\begin{align}
	E\big(\boldsymbol{\varphi}(t)\big)\geq\frac{1}{2}\int_{\Omega}|\nabla\varphi(t)|^{2}\,\mathrm{d}x
	+\frac{1}{2}\int_{\Gamma}|\nabla_{\Gamma}\psi(t)|^{2}\,\mathrm{d}S -c_{1}\big(|\Omega|+|\Gamma|\big),\quad\text{for a.a. }t\geq0.
	 \label{e-lower}
\end{align}
Combining it with \eqref{energyin}, recalling the Poincar\'e-Wirtinger inequalities \eqref{bulkpoin}, \eqref{surfacepoin}, we get
\begin{align*}
\int_{0}^{t}\Big(\|\partial_t\varphi(s)\|_{V_{0}^\ast}^{2} +\|\partial_t\psi(s)\|_{V_{\Gamma,0}^\ast}^{2}\Big)\,\mathrm{d}s\leq E(\boldsymbol{\varphi}_{0})+c_{1}\big(|\Omega|+|\Gamma|\big)
\end{align*}
and
\begin{align*}
	\Vert\boldsymbol{\varphi}(t)\Vert_{\mathcal{H}_{K}^{1}}^{2}&\leq 2(C_{p}^{2}+1)\|\varphi(t)-m_{0}\|_{V_{0}}^{2}+2|\Omega||m_{0}|^{2}\\
	&\quad+2(C_{p}^{2}+1)\|\psi(t)-m_{\Gamma0}\|_{V_{\Gamma,0}}^{2}+2|\Gamma||m_{\Gamma0}|^{2}\\
	&\leq4(C_{p}^{2}+1)E(\boldsymbol{\varphi}_{0})+2(C_{p}^{2}+1)c_{1}\big(|\Omega|+|\Gamma|\big)
	+2|\Omega||m_{0}|^{2}+2|\Gamma||m_{\Gamma0}|^{2}
\end{align*}
for almost all $t\geq0$. Since the right-hand side is independent of $t$, then using the Lebesgue monotone convergence theory, we obtain \eqref{eq3.7}. Finally, by the definition of operators $\mathcal{N}_{\Omega}$,  $\mathcal{N}_{\Gamma}$ and \eqref{eq2.7}, \eqref{eq2.10}, we have
\begin{align}
	&\mathcal{N}_{\Omega}\big(\partial_t\varphi(t)\big) =-\mathbf{P}_{\Omega}\big(\mu(t)\big)\quad\text{in } V_{0}, \notag\\%\label{equi}\\
	&\mathcal{N}_{\Gamma}\big(\partial_t\psi(t)\big) =-\mathbf{P}_{\Gamma}\big(\theta(t)\big)\quad\text{in } V_{\Gamma,0},\notag %\label{equi'}
\end{align}
for almost all $t\geq0$, which together with \eqref{eq3.7} allows us to conclude \eqref{eq3.7'}.
\end{proof}

\begin{lemma}
\label{time-deri}
Let $\boldsymbol{\varphi}$ be the global weak solution obtained in Proposition \ref{weakexist}. For any given $\tau>0$, there exists a positive constant $M_{2}$ such that
\begin{align}
\Vert\partial_t\varphi\Vert_{L^{\infty}(\tau,+\infty;V_{0}^\ast)} +\Vert\partial_t\psi\Vert_{L^{\infty}(\tau,+\infty;V_{\Gamma,0}^\ast)}
+\int_{t}^{t+1}\|\partial_t\boldsymbol{\varphi}(s)\|^{2}_{\mathcal{H}^{1}}\mathrm{d}s\leq M_{2},\quad\forall\, t\geq\tau.
\label{eq3.8}
\end{align}
\end{lemma}
\begin{proof}
Since the value of $K\in[0,+\infty)$ will not play a role in the proof, the results follows a standard approximating procedure using \eqref{appro1} (cf. \cite{AW,LW}). The details are omitted.
\end{proof}

\begin{proposition}[Energy equality]
	Let $\boldsymbol{\vp}$ be the global weak solution obtained in Proposition \ref{weakexist}. We have
	\begin{align}
		\frac{\mathrm{d}}{\mathrm{d}t}E\big(\boldsymbol{\varphi}(t)\big)
        +\Vert\partial_t\varphi(t)\Vert_{V_{0}^*}^{2}
		+\Vert\partial_t\psi(t)\Vert_{V_{\Gamma,0}^*}^{2}=0,
    \quad\text{for a.a. }t>0,\label{4.1}
	\end{align}
	and
	\begin{align}
		E\big(\boldsymbol{\varphi}(t)\big)
        +\int_{0}^{t}\Big(\Vert\partial_t\varphi(s)\Vert_{V_{0}^*}^{2}
		+\Vert\partial_t\psi(s)\Vert_{V_{\Gamma,0}^*}^{2}\Big)\,\mathrm{d}s =E(\boldsymbol{\varphi}_{0}),
        \quad\forall\,t\geq0.
\label{eq4.1b}
	\end{align}
\end{proposition}
\begin{proof}
	For $s\geq\tau>0$, taking $z=\mathcal{N}_{\Omega}\big(\partial_t\varphi(s)\big)$ in \eqref{eq2.7}, we obtain
	\begin{align}
		0&=\big\langle\partial_t\varphi(s),\mathcal{N}_{\Omega}\big(\partial_t\varphi(s)\big)\big\rangle_{V_{0}^*,V_{0}}
		+\int_{\Omega}\nabla\mu(s)\cdot\nabla\mathcal{N}_{\Omega}\big(\partial_t\varphi(s)\big)\,\mathrm{d}x\notag\\
		&=\Vert\partial_t\varphi(s)\Vert^{2}_{V_{0}^*}+\int_{\Omega}\mu(s)\partial_t\varphi(s)\,\mathrm{d}x\notag\\
		&=\Vert\partial_t\varphi(s)\Vert^{2}_{V_{0}^*}+\int_{\Omega}\big(-\Delta\varphi(s)+\beta\big(\varphi(s)\big)
		+\pi\big(\varphi(s)\big)\big)\partial_t\varphi(s)\,\mathrm{d}x.\notag		
	\end{align}
	Similarly, taking $z_{\Gamma}=\mathcal{N}_{\Gamma}\big(\partial_t\psi(s)\big)$ in \eqref{eq2.10}, we obtain
	\begin{align}
	0&=\big\langle\partial_{t}\psi(s),\mathcal{N}_{\Gamma} \big(\partial_t\psi(s)\big)\big\rangle_{V_{\Gamma,0}^*, V_{\Gamma,0}}	+\int_{\Gamma}\nabla_{\Gamma}\theta(s)\cdot\nabla_{\Gamma}\mathcal{N}_{\Gamma} \big(\partial_t\psi(s)\big)\,\mathrm{d}S \notag\\
	&=\Vert\partial_t\psi(s)\Vert^{2}_{V_{\Gamma,0}^*}
       + \int_{\Gamma}\theta(s)\partial_t\psi(s)\,\mathrm{d}S\notag\\
	&=\Vert\partial_t\psi(s)\Vert^{2}_{V_{\Gamma,0}^*}
      +\int_{\Gamma}\big(\partial_{\mathbf{n}}\varphi(s)	-\Delta_{\Gamma}\psi(s)+\beta_{\Gamma}\big(\psi(s)\big) +\pi_{\Gamma}\big(\psi(s)\big)\big)\partial_t\psi(s)\,\mathrm{d}S.\notag
	\end{align}
	Summing up the above two equalities, using the chain rule of the subdifferential (see, e.g., \cite{R.E.S}), we obtain
	\begin{align*}
		\frac{\mathrm{d}}{\mathrm{d}t}E\big(\boldsymbol{\varphi}(t)\big)+\Vert\partial_t\varphi(t)\Vert_{V_{0}^*}^{2}
		+\Vert\partial_t\psi(t)\Vert_{V_{\Gamma,0}^*}^{2}=0,\quad\text{for a.a. }t\geq\tau.
	\end{align*}
	For $s,t\geq\tau$, integrating over the interval $(s,t)$, we find
	\begin{align}
		E\big(\boldsymbol{\varphi}(t)\big)-E\big(\boldsymbol{\varphi}(s)\big)
		=-\int_{s}^{t}\big(\Vert\partial_t\varphi(\eta)\Vert_{V_{0}^*}^{2}
		+\Vert\partial_t\psi(\eta)\Vert_{V_{\Gamma,0}^*}^{2}\big)\,\mathrm{d}\eta.
       \label{positive}
	\end{align}
	Hence, for any fixed $\tau>0$, the mapping $t\mapsto E\big(\boldsymbol{\varphi}(t)\big)$ is absolutely continuous and non-increasing for all $t\geq\tau$. Since $\tau>0$ is arbitrary, the energy identity \eqref{eq4.1b} holds for almost all $t>0$. It follows from the energy inequality \eqref{energyin} that $\limsup_{s\to0}E(\boldsymbol{\varphi}(s))\leq E(\boldsymbol{\varphi}(0))$. On the other hand, from Remark \ref{V-conti}, the lower weak semicontinuity of norms and Lebesgue's dominated convergence theorem, we have $\liminf_{s\to0}E(\boldsymbol{\varphi}(s))\geq E(\boldsymbol{\varphi}(0))$. As a result, it holds $\lim_{s\to0}E(\boldsymbol{\varphi}(s))= E(\boldsymbol{\varphi}(0))$. Thus, we can pass to the limit $s\to0$ in \eqref{positive} to conclude the energy equality \eqref{eq4.1b}.
\end{proof}

\begin{remark}\rm
	Since  the mapping $t\mapsto E(\boldsymbol{\vp}(t))$ is absolutely continuous for all $t\geq0$, then applying the same argument as in \cite[Proposition 4.1]{FW}, we have $\boldsymbol{\vp}\in C([0,+\infty);\mathcal{H}_{K}^{1})$.
\end{remark}

As a consequence of \eqref{e-lower} and \eqref{eq4.1b}, we can conclude

\begin{corollary}
	\label{lowerbound}
	Let $\boldsymbol{\vp}$ be the global weak solution obtained in Proposition \ref{weakexist}. There exists some constant $E_{\infty}\in\mathbb{R}$ such that
	\begin{align*}
		\lim_{t\to+\infty}E(\boldsymbol{\vp}(t))=E_{\infty}\quad\text{and}
		\quad\lim_{t\to+\infty}\int_{t}^{t+1}\big(\|\nabla\mu(s)\|_{H}^{2}
		+\|\nabla_{\Gamma}\theta(s)\|_{H_{\Gamma}}^{2}\big)\,\mathrm{d}s=0.
	\end{align*}
\end{corollary}

\subsection{High-order regularity}
\begin{lemma}
Let $(\boldsymbol{\varphi}, \boldsymbol{\mu})$ be the global weak solution obtained in Proposition \ref{weakexist}.
For any given $\tau>0$, there exists a positive constant $M_{3}$ such that
\begin{align}
	&\Vert\beta(\varphi)\Vert_{L^{\infty}(\tau,+\infty;L^{1}(\Omega))}
	+\Vert\beta_{\Gamma}(\psi)\Vert_{L^{\infty}(\tau,+\infty;L^{1}(\Gamma))}\leq M_{3},\label{lv4.12}\\
	&\Vert\boldsymbol{\mu}\Vert_{L^{\infty}(\tau,+\infty;\mathcal{H}^{1})}\leq M_{3}.\label{lv4.13}
\end{align}
\end{lemma}
\begin{proof}
	The case $K=0$ has been treated in \cite[Lemma 4.1]{MW}. Concerning the case with $K\in(0,+\infty)$, from  \eqref{eq2.9} and \eqref{eq3.7}, we find
	\begin{align*}
		\|\partial_{\mathbf{n}}\varphi\|_{L^{\infty}(\tau,+\infty;H_{\Gamma})}
		=\frac{1}{K}\|\psi-\varphi\|_{L^{\infty}(\tau,+\infty;H_{\Gamma})}\leq C.
	\end{align*}
	With this simple observation, we can prove \eqref{lv4.12}, \eqref{lv4.13} by an argument similar to that in \cite[Lemma 4.1]{MW}.
\end{proof}

\begin{remark}\rm
	Thanks to \eqref{eq3.8} and \eqref{lv4.13}, we can regard \eqref{eq2.7} and \eqref{eq2.10} as elliptic problem for $\mu$ and $\theta$, respectively, that is
	\begin{align*}
	&\int_{\Omega}\nabla\mu(t)\cdot\nabla z\,\mathrm{d}x=-(\partial_t\varphi(t),z)_{H},
	\quad\text{for all }z\in V\text{ and a.a. }t\in[\tau,+\infty),\\
	&\int_{\Gamma}\nabla_{\Gamma}\theta(t)\cdot\nabla_{\Gamma} z_{\Gamma}\,\mathrm{d}S
	=-(\partial_t\psi(t),z_{\Gamma})_{H_{\Gamma}},
	\quad\text{for all }z_{\Gamma}\in V_{\Gamma}\text{ and a.a. }t\in[\tau,+\infty).
	\end{align*}
	From the elliptic regularity theorem, we find $\boldsymbol{\mu}\in L^{2}_{\text{uloc}}(\tau,+\infty;\mathcal{H}^{2})$. Since $\tau>0$ is arbitrary, this implies that $(\boldsymbol{\varphi},\boldsymbol{\mu})$ becomes a strong solution of problem \eqref{model} on $(0,+\infty)$.
\end{remark}

\begin{lemma}\label{regh2}
	Let $\boldsymbol{\varphi}$ be the global weak solution obtained in Proposition \ref{weakexist}.
	
	(1) In two dimensions, for any given $\tau>0$, $p\in[2,+\infty)$, there exists a positive constant $\widetilde{C}_{0}$ such that
	\begin{align}
		\Vert\beta(\varphi)\Vert_{L^{\infty}(\tau,+\infty;L^{p}(\Omega))}
		+\Vert\beta(\psi)\Vert_{L^{\infty}(\tau,+\infty;L^{p}(\Gamma))}
		\leq\widetilde{C}_{0}\sqrt{p},\label{eq4.4}
	\end{align}
	where $\widetilde{C}_{0}$ is independent of $p$.
	
	(2) In three dimensions, for any given $\tau>0$, there exists a positive constant $\widetilde{C}_1$ such that
	\begin{align}
		\Vert\beta(\varphi)\Vert_{L^{\infty}(\tau,+\infty;L^{p}(\Omega))} +\Vert\beta(\psi)\Vert_{L^{\infty}(\tau,+\infty;L^{p}(\Gamma))}\leq \widetilde{C}_1,\notag
	\end{align}
	where $p\in[1,6]$ and the constant $\widetilde{C}_{1}$ may depend on $p$.
\end{lemma}
\begin{proof}
In this proof, we adapt the notations which have been used in the proof of \cite[Lemma 3.5]{FW}. For each integer $k\geq 2$, we define the Lipschitz continuous function $h_{k}:\mathbb{R}\to\mathbb{R}$ by
\begin{align*}
	h_{k}(r):=
	\begin{cases}
	-1+\frac{1}{k}&\text{if }r<-1+\frac{1}{k},\\
	r&\text{if }-1+\frac{1}{k}\leq r\leq1-\frac{1}{k},\\
	1-\frac{1}{k}&\text{if }r>1-\frac{1}{k}.	
	\end{cases}
\end{align*}
For $s\geq\tau$, define
\begin{align*}
	\varphi_{k}(s):=h_{k}\circ\varphi(s),\quad\psi_{k}(s):=h_{k}\circ\psi(s).
\end{align*}
Then, we have $\boldsymbol{\varphi}_{k}:=(\varphi_{k},\psi_{k})\in C([\tau,+\infty);\mathcal{H}_{K}^{1})$ for any $\tau>0$ and
\begin{align*}
	\nabla\varphi_{k}=\nabla\varphi\,\chi_{[-1+\frac{1}{k},1-\frac{1}{k}]}(\varphi),
	\quad\nabla_{\Gamma}\psi_{k}=\nabla_{\Gamma}\psi\,\chi_{[-1+\frac{1}{k},1-\frac{1}{k}]}(\psi),
\end{align*}
where $\chi_{[-1+\frac{1}{k},1-\frac{1}{k}]}$ is the characteristic function of $[-1+1/k,1-1/k]$ defined by
\begin{align*}
	\chi_{[-1+\frac{1}{k},1-\frac{1}{k}]}(r):=
	\begin{cases}
		0,&\text{if }r\leq-1+\frac{1}{k},\\
		1,&\text{if }-1+\frac{1}{k}\leq r\leq1-\frac{1}{k},\\
		0,&\text{if }r\geq1-\frac{1}{k}.
	\end{cases}
\end{align*}
For any $k\geq 2$ and $p\geq2$, we see that $|\beta(\varphi_{k})|^{p-2}\beta(\varphi_{k})\in C([\tau,+\infty);V)$ and $|\beta(\psi_{k})|^{p-2}\beta(\psi_{k})\in C([\tau,+\infty);V_{\Gamma})$ are well-defined with
\begin{align*}
\nabla\big(|\beta(\varphi_{k})|^{p-2}\beta(\varphi_{k})\big)&=(p-1)|\beta(\varphi_{k})|^{p-2}\beta'(\varphi_{k})\nabla\varphi_{k},\\
\nabla_{\Gamma}\big(|\beta(\psi_{k})|^{p-2}\beta(\psi_{k})\big)&=(p-1)|\beta(\psi_{k})|^{p-2}\beta'(\psi_{k})\nabla_{\Gamma}\psi_{k}.
\end{align*}
Multiplying \eqref{eq2.8} by $|\beta(\varphi_{k}(s))|^{p-2}\beta(\varphi_{k}(s))$ and integrating over $\Omega$, in a similar manner, multiplying \eqref{eq2.11} by $|\beta(\psi_{k}(s))|^{p-2}\beta(\psi_{k}(s))$ and integrating over $\Gamma$, using integration by parts and then adding the resultants together, we have
\begin{align}
	&\int_{\Omega}|\beta(\varphi_{k}(s))|^{p-2}\beta(\varphi_{k}(s))\beta(\varphi(s))\,\mathrm{d}x
	+\int_{\Gamma}|\beta(\psi_{k}(s))|^{p-2}\beta(\psi_{k}(s))\beta_{\Gamma}(\psi(s))\,\mathrm{d}S\notag\\
	&\quad=-(p-1)\int_{\Omega}|\beta(\varphi_{k}(s))|^{p-2}\beta'(\varphi_{k}(s))\nabla\varphi_{k}(s)\cdot\nabla\varphi(s)\,\mathrm{d}x\notag\\
	&\qquad-(p-1)\int_{\Gamma}|\beta(\psi_{k}(s))|^{p-2}\beta'(\psi_{k}(s))\nabla_{\Gamma}\psi_{k}(s)\cdot\nabla_{\Gamma}\psi(s)\,\mathrm{d}S\notag\\
	&\qquad+\int_{\Omega}\widetilde{\mu}(s)|\beta(\varphi_{k}(s))|^{p-2}\beta(\varphi_{k}(s))\,\mathrm{d}x
	+\int_{\Gamma}\widetilde{\theta}(s)|\beta(\psi_{k}(s))|^{p-2}\beta(\psi_{k}(s))\,\mathrm{d}S\notag\\
	&\qquad+\int_{\Gamma}\partial_{\mathbf{n}}\varphi(s)\big(|\beta(\varphi_{k}(s))|^{p-2}\beta(\varphi_{k}(s))
	-|\beta(\psi_{k}(s))|^{p-2}\beta(\psi_{k}(s))\big)\,\mathrm{d}S,\label{bbb}
\end{align}
with $\widetilde{\mu}:=\mu-\pi(\varphi)\in L^{\infty}(\tau,+\infty;V)$ and $\widetilde{\theta}:=\theta-\pi_{\Gamma}(\psi)\in L^{\infty}(\tau,+\infty;V_{\Gamma})$.

When $K=0$, we note that $\big(|\beta(\varphi_{k})|^{p-2}\beta(\varphi_{k})\big)|_{\Gamma}=|\beta(\psi_{k})|^{p-2}\beta(\psi_{k})\in C([\tau,+\infty);V_{\Gamma})$, then the proof can be carried out in the same way as \cite[Lemma 3.4]{LvWu-2}. For the case $K\in(0,+\infty)$, since the boundary condition \eqref{eq2.9} does not allow us to derive the following relation:
$$|\beta(\varphi_{k})|^{p-2}\beta(\varphi_{k})=|\beta(\psi_{k})|^{p-2}\beta(\psi_{k})\quad\text{on }\Sigma,$$
we need to treat the last term  $\int_{\Gamma}\partial_{\mathbf{n}}\varphi\big(|\beta(\varphi_{k})|^{p-2}\beta(\varphi_{k})-|\beta(\psi_{k})|^{p-2} \beta(\psi_{k})\big)\,\mathrm{d}S$ on the right-hand side of \eqref{bbb}. By $\mathbf{(A1)}$, we see that for $p\geq2$, it holds
$$\big(|\beta(r)|^{p-2}\beta(r)\big)'=(p-1)|\beta(r)|^{p-2}\beta'(r)\geq0,\quad\forall\,r\in(-1,1),$$
which implies that the function $|\beta(r)|^{p-2}\beta(r)$ is monotone increasing on $(-1,1)$. By the construction of the Lipschitz continuous function $h_{k}$, we see that $h_{k}:\mathbb{R}\rightarrow(-1,1)$ is monotone increasing as well. Hence, we can conclude that the function $|\beta(h_{k}(r))|^{p-2}\beta(h_{k}(r))$ is also monotone increasing on $\mathbb{R}$. Taking the boundary condition \eqref{eq2.9} into account, it holds
\begin{align*}
&\int_{\Gamma}\partial_{\mathbf{n}}\varphi\big(|\beta(\varphi_{k})|^{p-2}\beta(\varphi_{k})
-|\beta(\psi_{k})|^{p-2}\beta(\psi_{k})\big)\,\mathrm{d}S\\
&\quad=\frac{1}{K}\int_{\Gamma}\big(\psi-\varphi\big)\big(|\beta(\varphi_{k})|^{p-2}\beta(\varphi_{k})
-|\beta(\psi_{k})|^{p-2}\beta(\psi_{k})\big)\,\mathrm{d}S\leq0.
\end{align*}
Then, we can conclude Lemma \ref{regh2} following the argument in \cite[Lemma 3.4]{LvWu-2}.
\end{proof}

Higher order estimates can be obtained using the argument like in \cite{FW,MW,LvWu-2}.
\begin{lemma}
\label{higher-order}
Let $(\boldsymbol{\varphi},\boldsymbol{\mu})$ be the global weak solution obtained in Proposition \ref{weakexist}. For any given $\tau>0$, there exists a positive constant $M_{4}$ such that
\begin{align}
	&\|\boldsymbol{\varphi}\|_{L^{\infty}(\tau,+\infty;\mathcal{H}_{K}^{2})}
	+\|\beta_{\Gamma}(\psi)\|_{L^{\infty}(\tau,+\infty;H_{\Gamma})}\leq M_{4}.\label{eq3.21}
\end{align}
\end{lemma}

\begin{remark}\rm
	\label{V-conti-high}
From Lemmas \ref{time-deri} and \ref{higher-order}, we find that $\boldsymbol{\varphi}\in C((0,+\infty);\mathcal{H}^{2r})$ for any $r\in(3/4,1)$. Thanks to the Sobolev embedding theorem, for any given $\tau>0$, it holds $\varphi\in C(\overline{\Omega}\times[\tau,+\infty))$ and $\psi\in C(\Gamma\times[\tau,+\infty))$, moreover,
\begin{align}
\|\varphi(t)\|_{C(\overline{\Omega})}\leq1,
\quad\|\psi(t)\|_{C(\Gamma)}\leq1,\quad\forall\,t\geq\tau.\label{conti1}
\end{align}
\end{remark}

\section{Separation from Pure States}
\setcounter{equation}{0}
In this section, we first establish Theorem \ref{intantaneous} on the instantaneous strict separation property in two dimensions under the assumption $\mathbf{(A5a)}$ or $\mathbf{(A5b)}$. Subsequently, we prove Theorem \ref{eventual} on the eventual strict separation property in both two and three dimensions, without relying on those additional assumptions.

\subsection{Instantaneous separation in two dimensions}
For any given $\tau>0$, we prove that the global weak solutions to problem \eqref{model} will stay uniformly away from $\pm1$ for all $t\geq \tau$. To this aim, we first establish the separation property by extending the direct method in \cite[Theorem 3.1]{GGG} for the Cahn-Hilliard equation subject to homogeneous Neumann boundary conditions, under the assumption $\mathbf{(A5a)}$. Then, we give a second proof based on a suitable De Giorgi's iteration scheme inspired by \cite{GP}, under the milder assumption $\mathbf{(A5b)}$  (cf. \cite{LvWu-2}).

\begin{proof}[\textbf{Proof of Theorem \ref{intantaneous}: the case with $\mathbf{(A5a)}$}]
By Fubini-Tonelli's theorem and \eqref{eq4.4}, we now compute for any $k\in\mathbb{Z}^+$ such that
	\begin{align}
		\sup_{t\geq\tau}\int_{\Omega}\Big|e^{C_{\sharp}|\beta(\varphi(t))|^{\gamma_{\sharp}}}\Big|^{k}\,\mathrm{d}x
		&=\sup_{t\geq\tau}\int_{\Omega}e^{C_{\sharp}k|\beta(\varphi(t))|^{\gamma_{\sharp}}}\,\mathrm{d}x\notag\\
		&=\sup_{t\geq\tau}\int_{\Omega}\sum_{p=0}^{\infty}\frac{C_{\sharp}^{p}k^{p}\big|\beta\big(\varphi(t)\big)\big|^{\gamma_{\sharp}p}}{p!}\,\mathrm{d}x\notag\\
		&=\sum_{p=0}^{\infty}\frac{C_{\sharp}^{p}k^{p}}{p!}\sup_{t\geq\tau}\int_{\Omega}\big|\beta\big(\varphi(t)\big)\big|^{\gamma_{\sharp}p}\,\mathrm{d}x\notag\\
		&=\sum_{p=0}^{\infty}\frac{C_{\sharp}^{p}k^{p}}{p!}\sup_{t\geq\tau}\big\|\beta\big(\varphi(t)\big)\big\|_{L^{\gamma_{\sharp}p}(\Omega)}^{\gamma_{\sharp}p}\notag\\
		&\leq\sum_{p=0}^{\infty}\frac{\big(\widetilde{C}_{0}^{\gamma_{\sharp}}C_{\sharp}k\big)^{p}}{p!}(\gamma_{\sharp}p)^{\frac{\gamma_{\sharp}p}{2}}.\notag
	\end{align}
	Following the same argument as that for \cite[(26)]{GGG}, we get
	\begin{align}
		&\sup_{t\geq\tau}\int_{\Omega}\Big|e^{C_{\sharp}|\beta(\varphi(t))|^{\gamma_{\sharp}}}\Big|^{k}\,\mathrm{d}x\notag\\
		&\qquad\leq 1+\frac{1}{\sqrt{2\pi}}\left(\frac{\Big(\widetilde{C}_{0}^{\gamma_{\sharp}}C_{\sharp}k\gamma_{\sharp}^{\frac{\gamma_{\sharp}}{2}}e\Big)^{2\Big(\widetilde{C}_{0}^{\gamma_{\sharp}} C_{\sharp}k\gamma_{\sharp}^{\frac{\gamma_{\sharp}}{2}}e\Big)^{2}}}{\mathrm{ln}\Big(\widetilde{C}_{0}^{\gamma_{\sharp}}C_{\sharp}k\gamma_{\sharp}^{\frac{\gamma_{\sharp}}{2}}e\Big)} +\frac{1}{\big(1-\frac{\gamma_{\sharp}}{2}\big)\mathrm{ln}(2)} +e^{\frac{\big(1-\frac{\gamma_{\sharp}}{2}\big)\Big(\widetilde{C}_{0}^{\gamma_{\sharp}} C_{\sharp}k\gamma_{\sharp}^{\frac{\gamma_{\sharp}}{2}}e\Big)^{\frac{1}{1-\frac{\gamma_{\sharp}}{2}}}}{e}}\right).\notag
	\end{align}
Then we deduce from $\mathbf{(A5a)}$ that
	\begin{align}
		\sup_{t\geq\tau}\Vert\beta'\big(\varphi(t)\big)\Vert_{L^{p}(\Omega)}\leq C_{\beta}(p),\quad \forall\, p\geq 2,\label{4.48-1}
	\end{align}
	where

	\begin{align*}
		C_{\beta}(p)=\left(C_{\sharp}+\frac{C_{\sharp}}{\sqrt{2\pi}} \left(\frac{\Big(\widetilde{C}_{0}^{\gamma_{\sharp}}C_{\sharp}p \gamma_{\sharp}^{\frac{\gamma_{\sharp}}{2}}e\Big)^{2\Big(\widetilde{C}_{0}^{\gamma_{\sharp}} C_{\sharp}p\gamma_{\sharp}^{\frac{\gamma_{\sharp}}{2}}e\Big)^{2}}} {\mathrm{ln}\Big(\widetilde{C}_{0}^{\gamma_{\sharp}}C_{\sharp}p\gamma_{\sharp}^{\frac{\gamma_{\sharp}}{2}}e\Big)} +\frac{1}{\big(1-\frac{\gamma_{\sharp}}{2}\big)\mathrm{ln}(2)} +e^{\frac{\big(1-\frac{\gamma_{\sharp}}{2}\big)\Big(\widetilde{C}_{0}^{\gamma_{\sharp}} C_{\sharp}p\gamma_{\sharp}^{\frac{\gamma_{\sharp}}{2}}e\Big)^{\frac{1}{1- \frac{\gamma_{\sharp}}{2}}}}{e}}\right)\right)^{1/p}.
	\end{align*}
	Analogously, we have
	\begin{align}
		\sup_{t\geq\tau}\Vert\beta'\big(\psi(t)\big)\Vert_{L^{p}(\Gamma)}\leq C_{\beta}(p),\quad \forall\, p\geq 2.\label{4.48-2}
	\end{align}
	On the other hand, we infer from \eqref{eq3.21} and the Sobolev embedding theorem that
	\begin{align}
		\sup_{t\geq\tau}\Vert\varphi(t)\Vert_{W^{1,6}(\Omega)} +\sup_{t\geq\tau}\Vert\psi(t)\Vert_{W^{1,6}(\Gamma)}
		\leq C_{E}\Big(\sup_{t\geq\tau}\Vert\varphi(t)\Vert_{H^{2}(\Omega)}
		+\sup_{t\geq\tau}\Vert\psi(t)\Vert_{H^{2}(\Gamma)}\Big)\leq\widetilde{C}_{\tau},
\label{4.49}
	\end{align}
	where $C_{E}>0$ is related to the constant of embedding. By \eqref{eq4.4}, \eqref{4.48-1}, \eqref{4.48-2}, \eqref{4.49} and H\"{o}lder's inequality, we obtain
	\begin{align*}
		\sup_{t\geq\tau}\Vert\beta\big(\varphi(t)\big)\Vert_{W^{1,3}(\Omega)}
		+\sup_{t\geq\tau}\Vert\beta\big(\psi(t)\big)\Vert_{W^{1,3}(\Gamma)}
		\leq\widetilde{C}_{0}\sqrt{3}+C_{\beta}(6)\widetilde{C}_{\tau}.
	\end{align*}
	Since $W^{1,3}(\Omega)\hookrightarrow C(\overline{\Omega})$ and $W^{1,3}(\Gamma)\hookrightarrow C(\Gamma)$ (recall that $d=2$), it follows that
	\begin{align*}
		\sup_{t\geq\tau}\Vert\beta\big(\varphi(t)\big)\Vert_{L^{\infty}(\Omega)}
		+\sup_{t\geq\tau}\Vert\beta\big(\psi(t)\big)\Vert_{L^{\infty}(\Gamma)}\leq
		C_{E}^{*}\Big(\widetilde{C}_{0}\sqrt{3}+C_{\beta}(6)\widetilde{C}_{\tau}\Big),
	\end{align*}
	where $C_{E}^{*}>0$ is related to the constant of embedding. Thus, taking $$\delta_{2}=1-\beta^{-1}\left(C_{E}^{*}\left(\widetilde{C}_{0}\sqrt{3} +C_{\beta}(6)\widetilde{C}_{\tau}\right)\right),$$
	we arrive at the conclusion \eqref{2.1}. Here, $\beta^{-1}$ denotes the inverse function of $\beta$ (cf. $\mathbf{(A1)}$).
\end{proof}

\begin{proof}[\textbf{Proof of Theorem \ref{intantaneous}: the case with $\mathbf{(A5b)}$ }]
	We now apply the De Giorgi's iteration scheme for the equations of chemical potentials $\mu$ and $\theta$ as in \cite{LvWu-2}. Let $\tau>0$ be arbitrary but fixed positive time. Thanks to Remark \ref{V-conti-high}, $\boldsymbol{\varphi}$ is well defined for all $t\geq\tau$ and satisfies \eqref{conti1}. For $\delta\in(0,1)$, we introduce the sequence
	$$k_{n}=1-\delta-\frac{\delta}{2^{n}},\quad\forall\,n\in\mathbb{N},$$
	such that
	$$1-2\delta<k_{n}<k_{n+1}<1-\delta,\quad\forall\,n\geq1,\qquad k_{n}\to1-\delta\text{ as }n\to+\infty.$$
	For any $n\in\mathbb{N}$, we set
	$$\varphi_{n}(x,t)=(\varphi(x,t)-k_{n})^{+},\quad\psi_{n}(x,t)=(\psi(x,t)-k_{n})^{+}.$$
	By definition, it is obvious that $0\leq\varphi_{n},\psi_{n}\leq 2\delta$ (see \cite{GP}). For any $n\in\mathbb{N}$, multiplying \eqref{eq2.8} by $\varphi_{n}$ and integrating over $\Omega$, multiplying \eqref{eq2.11} by $\psi_{n}$ and integrating over $\Gamma$, using integration by parts and then adding the resultants together, we obtain
	\begin{align}
		&\Vert\nabla\varphi_{n}\Vert_{H}^{2}+\Vert\nabla_{\Gamma}\psi_{n}\Vert_{H_{\Gamma}}^{2}
		+\int_{\Omega}\beta(\varphi)\varphi_{n}\,\mathrm{d}x
		+\int_{\Gamma}\beta_{\Gamma}(\psi)\psi_{n}\,\mathrm{d}S\notag\\
		&\quad=\int_{\Omega}\mu\varphi_{n}\,\mathrm{d}x+\int_{\Gamma}\theta\psi_{n}\,\mathrm{d}S +\left(-\int_{\Omega}\pi(\varphi)\varphi_{n}\,\mathrm{d}x
		-\int_{\Gamma}\pi_{\Gamma}(\psi)\psi_{n}\,\mathrm{d}S\right)\notag\\
		&\qquad+\int_{\Gamma}\partial_{\mathbf{n}}\varphi\big(\varphi_{n}-\psi_{n}\big)\,\mathrm{d}S.\label{de}
	\end{align}
	For the case $K=0$, the last term on the right hand side of \eqref{de} simply vanishes. Consequently, the proof can be done by repeating the argument in \cite[Section 4.1]{LvWu-2}. For the case $K\in(0,+\infty)$, taking the boundary condition \eqref{eq2.9} into account, we find $$\int_{\Gamma}\partial_{\mathbf{n}}\varphi\big(\varphi_{n}-\psi_{n}\big)\,\mathrm{d}S=\frac{1}{K}\int_{\Gamma}(\psi-\varphi)(\varphi_{n}-\psi_{n})\,\mathrm{d}S\leq0.$$
With this observation, the remaining part of the proof is again the same as in \cite{LvWu-2}.
\end{proof}

\subsection{Eventual separation property}

In what follows, we establish the eventual separation property by studying the strict separation property of  elements in $\omega$-limit set.
%dynamical approach (cf. \cite{AW,GMS09,FW,LvWu-2}).
For any given number $a_{1},a_{2}\in(-1,1)$, set $\boldsymbol{a}:=(a_{1},a_{2})$. We introduce the phase space
\begin{align*}
\mathcal{Z}_{\boldsymbol{a}}=
\big\{\boldsymbol{\varphi}=(\varphi,\psi)\in \mathcal{H}_{K}^{1}:\,\langle\varphi\rangle_{\Omega}=a_{1},\ \langle\psi\rangle_{\Gamma}=a_{2},\  E\big(\boldsymbol{\varphi}\big)<+\infty\,\big\}.
\end{align*}
The metric d$_{\mathcal{Z}_{\boldsymbol{a}}}\big(\cdot,\cdot\big)$ on $\mathcal{Z}_{\boldsymbol{a}}$ is defined as follows:
\begin{align}
\text{d}_{\mathcal{Z}_{\boldsymbol{a}}}\big(\boldsymbol{\varphi}_{1},\boldsymbol{\varphi}_{2}\big)
&:=\Vert\boldsymbol{\varphi}_{1}-\boldsymbol{\varphi}_{2}
\Vert_{\mathcal{H}_{K}^{1}}+\left|\int_{\Omega}\widehat{\beta}(\varphi_{1})\,\mathrm{d}x
-\int_{\Omega}\widehat{\beta}(\varphi_{2})\,\mathrm{d}x\right|^{\frac{1}{2}}\notag\\
&\quad\ \ +\left|\int_{\Gamma}\widehat{\beta}_{\Gamma}(\psi_{1})\,\mathrm{d}S
-\int_{\Gamma}\widehat{\beta}_{\Gamma}(\psi_{2})\,\mathrm{d}S\right|^{\frac{1}{2}}
,\quad\forall\,\boldsymbol{\varphi}_{1},\boldsymbol{\varphi}_{2}\in \mathcal{Z}_{\boldsymbol{a}}\notag
\end{align}
and $\big(\mathcal{Z}_{\boldsymbol{a}},\text{d}_{\mathcal{Z}_{\boldsymbol{a}}}(\cdot,\cdot)\big)$ is thus a complete metric space.

Define the $\omega$-limit set
\begin{align}
\omega(\boldsymbol{\varphi}_{0}):=\Big\{\boldsymbol{\varphi}_{\infty}\in\mathcal{H}^{2r}
\cap \mathcal{Z}_{\boldsymbol{m}_{0}}:\,\exists\, t_{n}\nearrow+\infty\text{ such that }\boldsymbol{\varphi}(t_{n})\rightarrow\boldsymbol{\varphi}_{\infty}\text{ in }\mathcal{H}^{2r}\text{ as }n\rightarrow+\infty\Big\}\label{omegalimitset}
\end{align}
for some $r\in(3/4,1)$ and $\boldsymbol{m}_{0}:=(m_{0},m_{\Gamma0})$. Since $\boldsymbol{\varphi}\in L^{\infty}(\tau,+\infty;\mathcal{H}_{K}^{2})$ for any $\tau>0$, then $\{\boldsymbol{\varphi}(t)\}_{t\geq\tau}$ is relatively compact in $\mathcal{H}^{2r}$. Hence, $\omega(\boldsymbol{\varphi}_{0})$ is nonempty, connected and compact in $\mathcal{H}^{2r}$.
We consider the problem \eqref{model} in the time interval $(t_{n},t_{n}+1)$, where $\{t_{n}\}_{n\in\mathbb{N}}$ is the sequence in \eqref{omegalimitset} and we introduce for $t\in[0,1]$ the functions
\begin{align*}
	\overline{\boldsymbol{\varphi}}_{n}(t):=\boldsymbol{\varphi}(t_{n}+t),
	\quad	\overline{\boldsymbol{\mu}}_{n}(t):=\boldsymbol{\mu}(t_{n}+t).
\end{align*}
Without loss of generality, we may assume that $t_{n}\geq1$ for all $n\in\mathbb{Z}^+$. It follows from  \eqref{eq3.7} that
\begin{align*}
	\partial_t\varphi\in L^{2}(0,+\infty;V_{0}^\ast),\quad\partial_t\psi\in L^{2}(0,+\infty;V_{\Gamma,0}^\ast),
\end{align*}
which, together with Lebesgue's dominated convergence theorem, implies
\begin{align*}
\partial_t\overline{\varphi}_{n}\to0\text{ strongly in }L^{2}(0,1;V_{0}^\ast), \quad\partial_t\overline{\psi}_{n}\to0\text{ strongly in }L^{2}(0,1;V_{\Gamma,0}^\ast).
\end{align*}
Based on the regularity properties of weak solutions obtained in Section \ref{globalregularity}, we obtain the following uniform estimates with respect to $n\in\mathbb{Z}^+$:
\begin{align*}
&\|\overline{\boldsymbol{\varphi}}_{n}\|_{L^{\infty}(0,1;\mathcal{H}_{K}^{2})}
+\|\overline{\boldsymbol{\mu}}_{n}\|_{L^{\infty}(0,1;\mathcal{H}^{1})\cap L^{2}(0,1;\mathcal{H}^{2})}
+\|\boldsymbol{\beta}(\overline{\boldsymbol{\varphi}}_{n})\|_{L^{\infty}(0,1;\mathcal{L}^{2})}\leq C.
\end{align*}
By an argument similar to that in \cite[Section 6]{CFL}, we can conclude that  $\boldsymbol{\varphi}_{\infty}$ is the strong solution of the following stationary problem:
\begin{align}
\left\{
\begin{array}{ll}
\Delta\mu_{\infty}=0,&\text{in }\Omega,\\
\mu_{\infty}=-\Delta\varphi_{\infty}+\beta\big(\varphi_{\infty}\big)
+\pi\big(\varphi_{\infty}\big),&\text{in }\Omega, \\
\partial_{\mathbf{n}}\mu_{\infty}=0,&\text{on }\Gamma,\\
K\partial_{\mathbf{n}}\varphi_{\infty}=\psi_{\infty}-\varphi_{\infty},&\text{on }\Gamma,\\
\Delta_{\Gamma}\theta_{\infty}=0,&\text{on }\Gamma,\\
\theta_{\infty}=\partial_{\mathbf{n}}\varphi_{\infty}-\Delta_{\Gamma}\psi_{\infty}
+\beta_{\Gamma}\big(\psi_{\infty}\big)+\pi_{\Gamma}\big(\psi_{\infty}\big),&\text{on }\Gamma.
\end{array}\right.\label{strong}
\end{align}
Here, both $\mu_{\infty}$ and $\theta_{\infty}$ are constants, that is,
\begin{align}
&\mu_{\infty}=\frac{1}{|\Omega|}\Big(\int_{\Omega}\big(\beta(\varphi_{\infty})+\pi(\varphi_{\infty})\big)\,\mathrm{d}x
-\int_{\Gamma}\partial_{\mathbf{n}}\varphi_{\infty}\,\mathrm{d}S\Big),\label{mu-cons}\\
&\theta_{\infty}=\frac{1}{|\Gamma|}\int_{\Gamma}\Big(\partial_{\mathbf{n}}\varphi_{\infty}+\beta_{\Gamma}(\psi_{\infty})
+\pi_{\Gamma}(\psi_{\infty})\Big)\,\mathrm{d}S.\label{theta-cons}
\end{align}
Thus, problem \eqref{strong} simplifies to
\begin{align}
	\left\{
	\begin{array}{ll}
		\mu_{\infty}=-\Delta\varphi_{\infty}
		+\beta\big(\varphi_{\infty}\big)+\pi\big(\varphi_{\infty}\big),&\text{in }\Omega, \\
		K\partial_{\mathbf{n}}\varphi_{\infty}
		=\psi_{\infty}-\varphi_{\infty},&\text{on }\Gamma,\\
		\theta_{\infty}=\partial_{\mathbf{n}}\varphi_{\infty}-\Delta_{\Gamma}\psi_{\infty}
		+\beta_{\Gamma}\big(\psi_{\infty}\big)+\pi_{\Gamma}\big(\psi_{\infty}\big),&\text{on }\Gamma,
	\end{array}\right.\notag%\label{strong1}
\end{align}
with $\mu_{\infty}$, $\theta_{\infty}$ given in \eqref{mu-cons}, \eqref{theta-cons}, respectively. Finally, due to Corollary \ref{lowerbound}, we see that $E(\boldsymbol{\varphi}_{\infty})\equiv E_{\infty}$ on the $\omega$-limit set $\omega(\boldsymbol{\varphi}_{0})$.

\begin{proof}[\textbf{Proof of Theorem \ref{eventual}: the dynamic approach.}]
Based on the uniform estimates derived in Lemmas \ref{ener}--\ref{higher-order}, by the definition of the $\omega$-limit set $\omega(\boldsymbol{\varphi}_{0})$, we can conclude that, there exists a positive constant $M_{5}$ such that
$$|\mu_{\infty}|+|\theta_{\infty}|\leq M_{5}$$
for $\mu_{\infty}$, $\theta_{\infty}$ given in \eqref{mu-cons}, \eqref{theta-cons} associated with $\boldsymbol{\varphi}_{\infty}\in\omega(\boldsymbol{\varphi}_{0})$. Under assumptions $\mathbf{(A1)}$--$\mathbf{(A4)}$, by a similar argument as \cite[Lemma 4.1]{FW}, there exists an uniform constant $\widetilde{\delta}\in(0,1)$ such that, for every $\boldsymbol{\varphi}_{\infty}\in\omega(\boldsymbol{\varphi}_{0})$, it holds
\begin{align}
&-1+\widetilde{\delta}\leq\varphi_{\infty}\leq1-\widetilde{\delta},\quad\text{in }\Omega,\label{4.8}\\
&-1+\widetilde{\delta}\leq\psi_{\infty}\leq1-\widetilde{\delta},\quad\text{on }\Gamma.\label{4.9}
\end{align}
By the definition of $\omega(\boldsymbol{\varphi}_{0})$, we thus obtain
\begin{align*}
\lim_{t\rightarrow+\infty}\text{dist}\big(\mathcal{S}(t)\boldsymbol{\varphi}_{0},\omega(\boldsymbol{\varphi}_{0})\big)
=0\quad\text{in }\mathcal{H}^{2r},
\end{align*}
where the above distance is given by dist$\big(\boldsymbol{z},\omega(\boldsymbol{\varphi}_{0})\big)=\inf_{\boldsymbol{y}\in\omega(\boldsymbol{\varphi}_{0})}\Vert\boldsymbol{z}-\boldsymbol{y}\Vert_{\mathcal{H}^{2r}}$. For $r\in(d/4,1)$, $d\in\{2,3\}$, by the Sobolev embedding theorem, it holds $\mathcal{H}^{2r}\hookrightarrow C(\overline{\Omega})\times C(\Gamma)$. Hence, we infer from \eqref{4.8} and \eqref{4.9} that \eqref{2.1} holds  with the choice
$$\delta_{1}=\frac{1}{2}\widetilde{\delta},$$
where the constant $\widetilde{\delta}$ is determined as in \eqref{4.8} and \eqref{4.9}.
\end{proof}

\begin{proof}[\textbf{Proof of Theorem \ref{eventual}: via De Giorgi's iteration scheme.}]
Based on the dissipative nature of problem \eqref{model} (see \eqref{eq4.1b}), we can prove the eventual separation property using a suitable De Giorgi's iteration scheme as in \cite{LvWu-2}. Based on Corollary \ref{lowerbound}, by a similar calculation like for \cite[(4.20)]{LvWu-2}, we find
	$$\sup_{t\geq T_{\widetilde{\epsilon}}+1}\Big(\|\nabla\mu(t)\|_{H}^{2}+\|\nabla_{\Gamma}\theta(t)\|_{H_{\Gamma}}^{2}\Big)\leq C\widetilde{\epsilon},$$
	where the constant $\widetilde{\epsilon}\in(0,1)$ is arbitrary small and $T_{\widetilde{\epsilon}}\gg 1$ depends on $\widetilde{\epsilon}$. The only difference from the argument in \cite{LvWu-2} is that there exists an additional term $\int_{\Gamma}\partial_{\mathbf{n}}\varphi(\varphi_{n}-\psi_{n})\,\mathrm{d}S$ on the right-hand side of \cite[(4.22)]{LvWu-2} for the case $K\in(0,+\infty)$. Taking the boundary condition \eqref{eq2.9} into account, it holds
	$$\int_{\Gamma}\partial_{\mathbf{n}}\varphi(\varphi_{n}-\psi_{n})\,\mathrm{d}S=\frac{1}{K}\int_{\Gamma}(\psi-\varphi)(\varphi_{n}-\psi_{n})\,\mathrm{d}S\leq0.$$
	Hence, the additional term does not cause any trouble in the analysis and we can repeat word by word as we did in \cite{LvWu-2} to complete the proof.
\end{proof}

\section{Convergence to Equilibrium}
\setcounter{equation}{0}
\begin{lemma}
Under the assumptions of Theorem \ref{eventual}, there exists a positive constant $M_{6}$ such that
\begin{align}
\Vert\boldsymbol{\varphi}\Vert_{L^{\infty}(T_{\mathrm{SP}},+\infty;\mathcal{H}_{K}^{3})}\leq M_{6},\label{5.1}
\end{align}
where $T_{\mathrm{SP}}>0$ is determined as in Theorem \ref{eventual}.
\end{lemma}
\begin{proof}
The proof relies on the eventual strict property and the regularity theory of elliptic problems (cf. \cite[Lemma 5.1]{CFW}). Let us write the equations \eqref{eq2.8}--\eqref{eq2.11} as
\begin{align}
&-\Delta\varphi(t)=\mu(t)-\beta\big(\varphi(t)\big)-\pi\big(\varphi(t)\big):=h(t)\quad \text{a.e. in }\Omega,\notag\\
&K\partial_{\mathbf{n}}\varphi(t)=\psi(t)-\varphi(t)\quad \text{a.e. on }\Gamma,\notag\\
&\partial_{\mathbf{n}}\varphi(t)-\Delta_{\Gamma}\psi(t)+\psi(t)=\theta(t)-\beta_{\Gamma}\big(\psi(t)\big)
-\pi_{\Gamma}\big(\psi(t)\big)+\psi(t):=h_{\Gamma}(t)\quad \text{a.e. on }\Gamma,\notag
\end{align}
for almost all $t\geq T_{\mathrm{SP}}$. From the separation property \eqref{2.1}, $\mathbf{(A1)}$ and the estimate \eqref{eq3.7}, we can deduce that
\begin{align}
\big\Vert\beta\big(\varphi(t)\big)\big\Vert_{V}+\big\Vert\beta_{\Gamma}\big(\psi(t)\big)\big\Vert_{V_{\Gamma}}\leq C,\quad\text{for a.a. }t\geq T_{\mathrm{SP}}.\notag
\end{align}
Next, it follows from \eqref{lv4.13} that
\begin{align*}
\Vert\mu(t)\Vert_{V}+\Vert\theta(t)\Vert_{V_{\Gamma}}\leq C,\quad \text{for a.a. }t\geq T_{\mathrm{SP}}.
\end{align*}
As a consequence, we have
\begin{align*}
\Vert h(t)\Vert_{V}+\Vert h_{\Gamma}(t)\Vert_{V_{\Gamma}}\leq C,\quad \text{for a.a. }t\geq T_{\mathrm{SP}},
\end{align*}
which together with the elliptic regularity theory (see, e.g., \cite[Theorem 3.3]{KL}) yields \eqref{5.1}.
\end{proof}

By \eqref{eq3.7}, \eqref{5.1} and interpolation, we easily find
 $$\boldsymbol{\varphi}\in C([T_{\mathrm{SP}},+\infty);\mathcal{H}_{K}^{2}).$$
 The uniform estimate \eqref{5.1} and the compact embedding $\mathcal{H}^{3}\hookrightarrow\hookrightarrow\mathcal{H}^{2}$ also imply that the $\omega$-limit set $\omega(\boldsymbol{\varphi}_{0})$ is non-empty and compact in $\mathcal{H}^{2}$. As a result,
$$
\lim_{t\to+ \infty}\mathrm{dist}\big(\mathcal{S}(t)\boldsymbol{\varphi}_{0}, \omega(\boldsymbol{\varphi}_{0})\big)=0\quad\text{in }\mathcal{H}^{2}.
$$
Finally, to prove the $\omega$-limit set $\omega(\boldsymbol{\varphi}_{0})$ reduces to a singleton, we apply the \L ojasiewicz-Simon approach, see, e.g., for applications to the Cahn-Hilliard type equations \cite{AW,FW,GKY,GGM,LW,Wu}. For our current setting with bulk-surface coupling, the main tool is the following extended \L ojasiewicz-Simon inequality, whose proof can be found in the Appendix.

\begin{lemma}
\label{LS}
Suppose that the assumptions in Theorem \ref{eventual} are satisfied. Assume in addition, $\beta$, $\beta_{\Gamma}$ are real analytic on $(-1,1)$ and $\pi$, $\pi_{\Gamma}$ are real analytic on $\mathbb{R}$. Let $\boldsymbol{\varphi}_{\infty}\in\omega(\boldsymbol{\varphi}_{0})$, there exist constants $\varsigma^{*}\in(0,1/2)$ and $b^{\ast}>0$ such that
\begin{align}
\left\Vert\left(
\begin{array}{c}
\mathbf{P}_{\Omega}\big(-\Delta w+\beta(w)+\pi(w)\big)\\
\mathbf{P}_{\Gamma}\big(\partial_{\mathbf{n}}w-\Delta_{\Gamma}w_{\Gamma}+\beta_{\Gamma}(w_{\Gamma})
+\pi_{\Gamma}(w_{\Gamma})\big)
\end{array}\right)\right\Vert_{\mathcal{L}_{(0)}^{2}}
\geq |E(\boldsymbol{w})-E(\boldsymbol{\varphi}_{\infty})|^{1-\varsigma^{*}}\label{5.2}
\end{align}
for all $\boldsymbol{w}=(w,w_{\Gamma})\in\mathcal{W}_{K}^{2}$ satisfying $\Vert\boldsymbol{w}-\boldsymbol{\varphi}_{\infty}\Vert_{\mathcal{H}^{2}}\leq b^{*}$ and $\langle w\rangle_{\Omega}=\langle\varphi_{\infty}\rangle_{\Omega}$, $\langle w_{\Gamma}\rangle_{\Gamma}=\langle\psi_{\infty}\rangle_{\Gamma}$.
\end{lemma}
\begin{remark}\rm
As in \cite{FW}, \eqref{4.8} and \eqref{4.9} imply that all elements of $\omega\big(\boldsymbol{\varphi}_{0}\big)$ are uniformly separated from $\pm1$. Then we can take $b^{*}>0$  sufficiently small such that any element $\boldsymbol{w}\in\mathcal{W}_{K}^{2}$ satisfying $\Vert\boldsymbol{w}-\boldsymbol{\varphi}_{\infty}\Vert_{\mathcal{H}^{2}}\leq b^{*}$ is uniformly separated from $\pm1$ as well. In particular, this choice prevents the possible singularities of $\beta$, $\beta_{\Gamma}$.
\end{remark}

\begin{proof}[\textbf{Proof of Theorem \ref{equilibrium}}]
With the aid of Lemma \ref{LS}, the proof can be carried out in a procedure that now becomes standard, see, e.g.,  \cite{FW,GKY,GW} for similar arguments. Here, we just sketch the main steps. First, one can show the orbit of the unique global weak solution $\boldsymbol{\varphi}$ will fall into a small neighborhood of certain $\boldsymbol{\varphi}_{\infty}\in \omega(\bm{\varphi}_0)$ in $\mathcal{W}_{K}^{2}$ and stay there forever. This conclusion can be achieved by using the energy equality \eqref{4.1}, Lemma \ref{LS} together with a contradiction argument as in \cite{HJ}. Then, by the Poincar\'e-Wirtinger inequalities \eqref{bulkpoin}, \eqref{surfacepoin}, we find
$$
\|\partial_t\varphi\|_{V_{0}^\ast}+\|\partial_t \psi\|_{V_{\Gamma,0}^\ast}=\|\nabla\mu\|_{H}+\|\nabla_{\Gamma}\theta\|_{H_{\Gamma}}\geq C|E(\boldsymbol{\varphi})-E(\boldsymbol{\varphi}_{\infty})|^{1-\varsigma^{*}}.
$$
This inequality combined with the energy equality \eqref{eq4.1b} and the \L ojasiewicz-Simon inequality \eqref{5.2} leads to
$$
\int_0^{+\infty} \left(\|\partial_t\varphi(t)\|_{V_{0}^\ast}+\|\partial_t \psi(t)\|_{V_{\Gamma,0}^\ast}\right)\,\mathrm{d}t<+\infty.
$$
Hence, $(\varphi(t)-m_0,\psi(t)-m_{\Gamma0}) \to  ({\varphi}_{\infty}-m_0,\psi_\infty-m_{\Gamma_0})$ in $V_{0}^\ast\times V_{\Gamma,0}^\ast $ as $t\to+ \infty$. By \eqref{5.1} and interpolation, we get the convergence $\boldsymbol{\varphi}\to \boldsymbol{\varphi}_{\infty}$ in $\mathcal{H}^{2}$ as $t\to+\infty$. Finally, the convergence rate \eqref{conver-rate} follows from an argument similar to that in \cite{GW,LW,Wu}.
\end{proof}

\section{Double Obstacle Limit}
\setcounter{equation}{0}
In the final section, we study the double obstacle limit by passing to the limit $\Theta_{k}\rightarrow0$ in the logarithmic potential. In the following, we will denote $\boldsymbol{\varphi}_{\Theta_{k}}:=(\varphi_{\Theta_{k}},\psi_{\Theta_{k}})$, $\boldsymbol{\mu}_{\Theta_{k}}:=(\mu_{\Theta_{k}},\theta_{\Theta_{k}})$ as the weak solution to problem $(S_{\Theta_{k}})$ obtained in Proposition \ref{weakexist}, $k\in\mathbb{Z}^{+}$. For arbitrary but given final time $T\in(0,+\infty)$, we derive \emph{a priori} estimates with respect to $k\in\mathbb{Z}^{+}$. First of all, as a direct result of the energy equality \eqref{eq4.1b}, we have
\begin{lemma}
There exists a constant $C>0$, independent of $k\in\mathbb{Z}^{+}$, such that
\begin{align}
	\Vert\boldsymbol{\varphi}_{\Theta_{k}}\Vert_{L^{\infty}(0,T;\mathcal{H}_{K}^{1})}
	+\Vert\nabla\mu_{\Theta_{k}}\Vert_{L^{2}(0,T;H)}+\Vert\nabla_{\Gamma}\theta_{\Theta_{k}}\Vert_{L^{2}(0,T;H_{\Gamma})}
	\leq C.\label{kappauni1}
\end{align}
\end{lemma}

\begin{lemma}
	\label{kappa2}
There exists a constant $C>0$, independent of $k\in\mathbb{Z}^{+}$, such that
\begin{align}
	&\Vert \Theta_{k}f_{0}(\varphi_{\Theta_{k}})\Vert_{L^{2}(0,T;L^{1}(\Omega))}+\Vert\Theta_{k} f_{0}(\psi_{\Theta_{k}})\Vert_{L^{2}(0,T;L^{1}(\Gamma))}\leq C\big(1+\Vert\partial_{\mathbf{n}}\varphi_{\Theta_{k}}\Vert_{L^{2}(0,T;H_{\Gamma})}\big).\label{kappauni2}
\end{align}
Furthermore,
\begin{align}
	\Vert\mu_{\Theta_{k}}\Vert_{L^{2}(0,T;V)}+\Vert\theta_{\Theta_{k}}\Vert_{L^{2}(0,T;V)}\leq C\big(1+\Vert\partial_{\mathbf{n}}\varphi_{\Theta_{k}}\Vert_{L^{2}(0,T;H_{\Gamma})}\big).\label{kappauni3}
\end{align}
\end{lemma}
\begin{proof}
	Testing \eqref{eq2.8} by $\varphi_{\Theta_{k}}-\langle\varphi_{\Theta_{k}}\rangle_{\Omega}$, then
	\begin{align}
		&\underbrace{\int_{\Omega}\mu_{\Theta_{k}}(\varphi_{\Theta_{k}}-\langle\varphi_{\Theta_{k}}\rangle_{\Omega})\,\mathrm{d}x}_{I_{1}}\notag\\
		&\quad=\int_{\Omega}\mu_{\Theta_{k}}(\varphi_{\Theta_{k}}-\langle\varphi_{0,k}\rangle_{\Omega})\,\mathrm{d}x\notag\\
		&\quad=-\int_{\Gamma}\partial_{\mathbf{n}}\varphi_{\Theta_{k}}(\varphi_{\Theta_{k}}-\langle\varphi_{0,k}\rangle_{\Omega})\,\mathrm{d}S
		+\int_{\Omega}|\nabla\varphi_{\Theta_{k}}|^{2}\,\mathrm{d}x\notag\\
		&\qquad+\underbrace{\Theta_{k}\int_{\Omega}f_{0}(\varphi_{\Theta_{k}})(\varphi_{\Theta_{k}}
			-\langle\varphi_{0,k}\rangle_{\Omega})\,\mathrm{d}x}_{I_{2}}
			-\Theta_{c}\int_{\Omega}\varphi_{\Theta_{k}}(\varphi_{\Theta_{k}}-\langle\varphi_{0,k}\rangle_{\Omega})\,\mathrm{d}x.\label{lw1}
	\end{align}
	For the term $I_{1}$, by the H\"older's inequality and the Poincar\'e-Wirtinger inequality \eqref{bulkpoin}, it holds
	\begin{align*}
		I_{1}&=\int_{\Omega}(\mu_{\Theta_{k}}-\langle\mu_{\Theta_{k}}\rangle_{\Omega})\varphi_{\Theta_{k}}\,\mathrm{d}x\leq C\Vert\nabla\mu_{\Theta_{k}}\Vert_{H}\Vert\varphi_{\Theta_{k}}\Vert_{H}\leq C\Vert\nabla\mu_{\Theta_{k}}\Vert_{H}.
	\end{align*}
	For the term $I_{2}$, since $\sup_{k\in\mathbb{Z}^{+}}|\langle\varphi_{0,k}\rangle_{\Omega}|<1$, there exists a constant $r_{0}\in(0,1)$ such that
	\begin{align}
	-1+r_{0}\leq\langle\varphi_{0,k}\rangle_{\Omega}\leq1-r_{0}\quad\forall\,k\in\mathbb{Z}^{+}.\label{compact}
	\end{align}
	Recall that for any $r$, $m\in(-1,1)$ (cf. \cite[Proposition 4.3]{Mi}):
	\begin{align*}
		f_{0}(r)(r-m)\geq c_{m}|f_{0}(r)|-c_{m}',\quad c_{m}>0,\quad c_{m}'\geq0,
	\end{align*}
	where the constants $c_{m}$ and $c_{m}'$ depend continuously on $m$. Hence, for $m\in[-1+r_{0},1-r_{0}]$, there exist constants $\widetilde{c}_{r_{0}}>0$ and $\widetilde{c}_{r_{0}}'\geq0$, depending on $r_{0}\in(0,1)$ but independent of $m\in[-1+r_{0},1-r_{0}]$, such that
	\begin{align}
	f_{0}(r)(r-m)\geq \widetilde{c}_{r_{0}}|f_{0}(r)|-\widetilde{c}_{r_{0}}',\quad \forall\,r\in(-1,1),\quad\forall\,m\in[-1+r_{0},1-r_{0}].\label{uni-cons}
	\end{align}
	Then, by \eqref{compact} and \eqref{uni-cons}, it holds
	\begin{align*}
		I_{2}\geq \widetilde{c}_{r_{0}}\Theta_{k}\Vert f_{0}(\varphi_{\Theta_{k}})\Vert_{L^{1}(\Omega)}-\widetilde{c}_{r_{0}}'.
	\end{align*}
The other terms on the right-hand side of \eqref{lw1} can be controlled by usage of H\"older's inequality. Then, by \eqref{kappauni1}, we can conclude that
	\begin{align*}
		\Vert\Theta_{k} f_{0}(\varphi_{\Theta_{k}})\Vert_{L^{2}(0,T;L^{1}(\Omega))}\leq C(1+\Vert\partial_{\mathbf{n}}\varphi_{\Theta_{k}}\Vert_{L^{2}(0,T;H_{\Gamma})}).
	\end{align*}
	Testing \eqref{eq2.11} by $\psi_{\Theta_{k}}-\langle\psi_{\Theta_{k}}\rangle_{\Gamma}$, by a similar way, we obtain the second estimate of \eqref{kappauni2}. Integrating \eqref{eq2.8} over $\Omega$ and \eqref{eq2.11} over $\Gamma$, we have
	\begin{align}
		\Big\|\int_{\Omega}\mu_{\Theta_{k}}\,\mathrm{d}x\Big\|_{L^{2}(0,T)}+\Big\|\int_{\Gamma}\theta_{\Theta_{k}}\,\mathrm{d}S\Big\|_{L^{2}(0,T)}\leq C\big(1+\Vert\partial_{\mathbf{n}}\varphi_{\Theta_{k}}\Vert_{L^{2}(0,T;H_{\Gamma})}\big).\notag
	\end{align}
	Then, by Poincar\'e-Wirtinger inequalities \eqref{bulkpoin}, \eqref{surfacepoin} and \eqref{kappauni1}, we can conclude \eqref{kappauni3}. Therefore, we complete the proof of Lemma \ref{kappa2}.
\end{proof}

\begin{lemma}
	There exists a constant $C>0$, independent of $k\in\mathbb{Z}^{+}$, such that
	\begin{align}
		&\Vert\Theta_{k} f_{0}(\varphi_{\Theta_{k}})\Vert_{L^{2}(0,T;H)}+\Vert\Theta_{k} f_{0}(\psi_{\Theta_{k}})\Vert_{L^{2}(0,T;H_{\Gamma})}+\Vert\boldsymbol{\mu}_{\Theta_{k}}\Vert_{L^{2}(0,T;\mathcal{H}^{1})}\notag\\
		&\qquad+\Vert\boldsymbol{\varphi}_{\Theta_{k}}\Vert_{L^{2}(0,T;\mathcal{H}_{K}^{2})}
		+\Vert\partial_{t}\varphi_{\Theta_{k}}\Vert_{L^{2}(0,T;V_{0}^{\ast})}+\Vert\partial_{t}\psi_{\Theta_{k}}\Vert_{L^{2}(0,T;V_{\Gamma,0}^{\ast})}\leq C.\label{kappauni7}
	\end{align}
\end{lemma}

\begin{proof}
We just make estimates formally and this process can be made rigorously by a cut-off for $\varphi_{\Theta_{k}}$ and $\psi_{\Theta_{k}}$ (see Lemma \ref{regh2}). Testing \eqref{eq2.8} by $\Theta_{k} f_{0}(\varphi_{\Theta_{k}})$, we obtain
	\begin{align*}
		\int_{\Omega}\Theta_{k}^{2}|f_{0}(\varphi_{\Theta_{k}})|^{2}\,\mathrm{d}x
		&=\int_{\Omega}\big(\mu_{\Theta_{k}}+\Theta_{c}\varphi_{\Theta_{k}}\big)\Theta_{k}\ f_{0}(\varphi_{\Theta_{k}})\,\mathrm{d}x\\
		&\quad-\int_{\Omega}\Theta_{k} f_{0}'(\varphi_{\Theta_{k}})|\nabla\varphi_{\Theta_{k}}|^{2}\,\mathrm{d}x+\int_{\Gamma}\Theta_{k} f_{0}(\varphi_{\Theta_{k}})\partial_{\mathbf{n}}\varphi_{\Theta_{k}}\,\mathrm{d}S.
	\end{align*}
	Testing \eqref{eq2.11} by $\Theta_{k} f_{0}(\psi_{\Theta_{k}})$, it holds
	\begin{align*}
		\int_{\Gamma}\Theta_{k}^{2}|f_{0}(\psi_{\Theta_{k}})|^{2}\,\mathrm{d}S&=\int_{\Gamma}\big(\theta_{\Theta_{k}}
		+\Theta_{c}\psi_{\Theta_{k}}\big)\Theta_{k} f_{0}(\psi_{\Theta_{k}})\,\mathrm{d}S\\
		&\quad-\int_{\Gamma}\Theta_{k} f_{0}'(\psi_{\Theta_{k}})|\nabla_{\Gamma}\psi_{\Theta_{k}}|^{2}\,\mathrm{d}S-\int_{\Gamma}\Theta_{k} f_{0}(\psi_{\Theta_{k}})\partial_{\mathbf{n}}\varphi_{\Theta_{k}}\,\mathrm{d}S.
	\end{align*}
		Adding the above two equalities together, using Young's inequality and the monotonicity of $f_{0}$, we obtain
	\begin{align*}
		&\int_{\Omega}\Theta_{k}^{2}|f_{0}(\varphi_{\Theta_{k}})|^{2}\,\mathrm{d}x
		+\int_{\Gamma}\Theta_{k}^{2}|f_{0}(\psi_{\Theta_{k}})|^{2}\,\mathrm{d}S\\
		&\quad\leq C\big(\Vert\mu_{\Theta_{k}}\Vert_{H}^{2}+\Vert\theta_{\Theta_{k}}\Vert_{H_{\Gamma}}^{2}+1\big)
		+\Theta_{k}\int_{\Gamma}(f_{0}(\varphi_{\Theta_{k}})-f_{0}(\psi_{\Theta_{k}}))\partial_{\mathbf{n}}\varphi_{\Theta_{k}}\,\mathrm{d}S\\
		&\quad= C\big(\Vert\mu_{\Theta_{k}}\Vert_{H}^{2}+\Vert\theta_{\Theta_{k}}\Vert_{H_{\Gamma}}^{2}+1\big)
		+\chi(K)\Theta_{k}\int_{\Gamma}(f_{0}(\varphi_{\Theta_{k}})-f_{0}(\psi_{\Theta_{k}}))(\psi_{\Theta_{k}}-\varphi_{\Theta_{k}})\,\mathrm{d}S\\
		&\quad\leq C\big(\Vert\mu_{\Theta_{k}}\Vert_{H}^{2}+\Vert\theta_{\Theta_{k}}\Vert_{H_{\Gamma}}^{2}+1\big),
	\end{align*}
	which, together with \eqref{kappauni3}, implies that
	\begin{align}
	\Vert \Theta_{k}f_{0}(\varphi_{\Theta_{k}})\Vert_{L^{2}(0,T;H)}+\Vert\Theta_{k} f_{0}(\psi_{\Theta_{k}})\Vert_{L^{2}(0,T;H_{\Gamma})}\leq C\big(1+\Vert\partial_{\mathbf{n}}\varphi_{\Theta_{k}}\Vert_{L^{2}(0,T;H_{\Gamma})}\big).\notag%\label{L2L2}
	\end{align}
	Then, regarding \eqref{eq2.8}--\eqref{eq2.11} as a bulk-surface coupled elliptic system for $\boldsymbol{\varphi}_{\Theta_{k}}$, by the elliptic regularity theorem (cf. \cite[Theorem 3.3]{KL}), trace theorem, Ehrling's lemma (see Lemma \ref{Ehrling}) and \eqref{kappauni1}, it holds
	\begin{align*}
		\Vert\boldsymbol{\varphi}_{\Theta_{k}}\Vert_{L^{2}(0,T;\mathcal{H}_{K}^{2})}&\leq C\big(1+\Vert\partial_{\mathbf{n}}\varphi_{\Theta_{k}}\Vert_{L^{2}(0,T;H_{\Gamma})}\big)\\
		&\leq C\big(1+\Vert\varphi_{\Theta_{k}}\Vert_{L^{2}(0,T;H^{r}(\Omega))}\big)\\
		&\leq\frac{1}{2}	\Vert\boldsymbol{\varphi}_{\Theta_{k}}\Vert_{L^{2}(0,T;\mathcal{H}_{K}^{2})}+C,
	\end{align*}
	for some $r\in(3/2,2)$. Then, we can conclude the first four estimates in \eqref{kappauni7}. Finally, using \eqref{eq2.7}, \eqref{eq2.10}, we have
	\begin{align*}
		\Vert\partial_{t}\varphi_{\Theta_{k}}\Vert_{V_{0}^{\ast}}= \Vert\nabla\mu_{\Theta_{k}}\Vert_{H},\quad\Vert\partial_{t}\psi_{\Theta_{k}}\Vert_{V_{\Gamma,0}^{\ast}}= \Vert\nabla_{\Gamma}\theta_{\Theta_{k}}\Vert_{H_{\Gamma}},
	\end{align*}
	which, together with \eqref{kappauni1}, imply the last two estimates in \eqref{kappauni7}.
\end{proof}

\begin{proof}[\textbf{Proof of Theorem \ref{doubleobstacle}}]
	Thanks to the Banach-Alaoglu theorem, Aubin-Lions-Simon lemma (see Lemma \ref{ALS}) and uniform estimates \eqref{kappauni1}, \eqref{kappauni7}, we conclude that there exist functions $(\widetilde{\boldsymbol{\varphi}},\widetilde{\boldsymbol{\mu}},\widetilde{\boldsymbol{\xi}})$ and a subsequence of $\big\{(\boldsymbol{\varphi}_{\Theta_{k}},\boldsymbol{\mu}_{\Theta_{k}})\big\}_{k\in\mathbb{Z}^{+}}$ (not relabelled), such that the convergence results \eqref{conver1}--\eqref{conver6} hold. By \eqref{eq2.8}--\eqref{eq2.11}, it holds
	\begin{align}
	\int_{\Omega}\mu_{\Theta_{k}}\zeta\,\mathrm{d}x+\int_{\Gamma}\theta_{\Theta_{k}}\zeta_{\Gamma}\,\mathrm{d}S&=\int_{\Omega}\nabla\varphi_{\Theta_{k}}\cdot\nabla \zeta\,\mathrm{d}x+\int_{\Omega}\big(f_{0}(\varphi_{\Theta_{k}})-\Theta_{c}\varphi_{\Theta_{k}}\big)\zeta\,\mathrm{d}x\notag\\
	&\quad+\int_{\Gamma}\nabla_{\Gamma}\psi_{\Theta_{k}}\cdot\nabla_{\Gamma} \zeta_{\Gamma}\,\mathrm{d}S+\int_{\Gamma}\big(f_{0}(\psi_{\Theta_{k}})-\Theta_{c}\psi_{\Theta_{k}}\big)\zeta_{\Gamma}\,\mathrm{d}S\notag\\
	&\quad+\chi(K)\int_{\Gamma}(\psi_{\Theta_{k}}-\varphi_{\Theta_{k}})(\zeta_{\Gamma}-\zeta)\,\mathrm{d}S\label{muthetaweak}
	\end{align}
	for any $(\zeta,\zeta_{\Gamma})\in\mathcal{H}_{K}^{1}$. Passing to the limit $k\to+\infty$ in \eqref{eq2.7}, \eqref{eq2.10} and \eqref{muthetaweak}, we obtain
	\begin{align}
	&\langle\partial_{t}\widetilde{\varphi},z\rangle_{V',V}
	+\int_{\Omega}\nabla\widetilde{\mu}\cdot\nabla z\,\mathrm{d}x=0,&&\forall\,z\in V,\notag\\%\label{eq2.7'}\\
    &\langle\partial_{t}\widetilde{\psi},z_{\Gamma}\rangle_{V_{\Gamma}',V_{\Gamma}}
	+\int_{\Gamma}\nabla_{\Gamma}\widetilde{\theta}\cdot\nabla_{\Gamma}z_{\Gamma}\,\mathrm{d}S=0,&&\forall\,z_{\Gamma}\in V_{\Gamma},\notag%\label{eq2.10'}\\
		\end{align}
		for almost all $t\in(0,T)$, and
	\begin{align}
	\int_{\Omega}\widetilde{\mu}\zeta\,\mathrm{d}x+\int_{\Gamma}\widetilde{\theta}\zeta_{\Gamma}\,\mathrm{d}S&=\int_{\Omega}\nabla\widetilde{\varphi}\cdot\nabla \zeta\,\mathrm{d}x+\int_{\Omega}\big(\widetilde{\xi}-\Theta_{c}\widetilde{\varphi}\big)\zeta\,\mathrm{d}x\notag\\
	&\quad+\int_{\Gamma}\nabla_{\Gamma}\widetilde{\psi}\cdot\nabla_{\Gamma} \zeta_{\Gamma}\,\mathrm{d}S+\int_{\Gamma}\big(\widetilde{\xi}_{\Gamma}-\Theta_{c}\widetilde{\psi}\big)\zeta_{\Gamma}\,\mathrm{d}S\notag\\
	&\quad+\chi(K)\int_{\Gamma}(\widetilde{\psi}-\widetilde{\varphi})(\zeta_{\Gamma}-\zeta)\,\mathrm{d}S,\label{muthetaweak1}
	\end{align}
	for almost all $t\in(0,T)$ and all $(\zeta,\zeta_{\Gamma})\in\mathcal{H}_{K}^{1}$. Since $(\zeta,\zeta_{\Gamma})\in\mathcal{H}_{K}^{1}$ is arbitrary, we can easily derive from \eqref{muthetaweak1} that
	\begin{align}
			&\widetilde{\mu}=-\Delta\widetilde{\varphi}+\widetilde{\xi}-\Theta_{c}\widetilde{\varphi},&&\text{a.e. in }Q_{T},\notag\\%\label{eq2.8'}\\
			&K\partial_{\mathbf{n}}\widetilde{\varphi}=\widetilde{\psi}-\widetilde{\varphi},&&\text{a.e. on }\Sigma_{T},\notag\\%\label{eq2.9'}\\
			&\widetilde{\theta}=\partial_{\mathbf{n}}\widetilde{\varphi}-\Delta_{\Gamma}\widetilde{\psi}
			+\widetilde{\xi}_{\Gamma}-\Theta_{c}\widetilde{\psi},&&\text{a.e. on }\Sigma_{T}.\notag%\label{eq2.11'}\\
	\end{align}
		%Now, we consider the limit (formally) of $\frac{\Theta}{2}F_{0}$ as $\Theta\rightarrow0$, that is,
		%\begin{align*}
		%	\frac{\Theta}{2}F_{0}(r)\rightarrow I_{[-1,1]}(r)=\left\{
		%	\begin{array}{ll}
		%		0,&\text{if }r\in[-1,1],\\
		%		+\infty,&\text{else}.
		%	\end{array}
		%	\right.
		%\end{align*}
		%We note that
		%\begin{align*}
		%	\partial I_{[-1,1]}(r)=
		%	\begin{cases}
		%		0,&\text{if }r\in(-1,1),\\
		%		[0,+\infty),&\text{if }r=1,\\
		%		(-\infty,0],&\text{if }r=-1,\\
		%		\emptyset,&\text{else}.
		%	\end{cases}
		%\end{align*}
		Furthermore, by the strong convergences \eqref{conver1} and \eqref{conver4}, together with $\lim_{k\to+\infty}\|\boldsymbol{\varphi}_{0,k}-\boldsymbol{\varphi}_{0}\|_{\mathcal{H}^{1}}=0$, we can conclude that the initial conditions hold
		\begin{align*}
	\widetilde{\varphi}|_{t=0}=\varphi_{0}\ \text{ a.e. in }\Omega,\qquad\widetilde{\psi}|_{t=0}=\psi_{0}\ \text{ a.e. on }\Gamma.
		\end{align*}
		
		Now, following the idea in \cite{Abels11}, we prove that
		\begin{align*}
			\widetilde{\xi}\in \partial I_{[-1,1]}(\widetilde{\varphi})\quad\text{a.e. in }Q_{T}\quad\text{and}\quad\widetilde{\xi}_{\Gamma}\in \partial I_{[-1,1]}(\widetilde{\psi})\quad\text{a.e. on }\Sigma_{T}.
		\end{align*}
		Due to \eqref{conver1}, \eqref{conver4} and
		\begin{align*}
			\Vert f\Vert_{L^{\infty}}\leq C\Vert f\Vert_{L^{2}}^{1-d/4}\,\Vert f\Vert_{H^{2}}^{d/4},
		\end{align*}
		we obtain
		\begin{align*}
			\begin{array}{ll}
			\varphi_{\Theta_{k}}\rightarrow\widetilde{\varphi}&\qquad\text{strongly in }L^{2}(0,T;C(\overline{\Omega}))\\
			\psi_{\Theta_{k}}\rightarrow\widetilde{\psi}&\qquad\text{strongly in }L^{2}(0,T;C(\Gamma))
			\end{array}
		\end{align*}
		as $k\to+\infty$. Therefore, up to a suitable subsequence, there hold
		\begin{align}
			&\varphi_{\Theta_{k}}(t)\rightarrow\widetilde{\varphi}(t)\quad\text{in }C(\overline{\Omega}),\notag\\
			&\psi_{\Theta_{k}}(t)\rightarrow\widetilde{\psi}(t)\quad\text{in }C(\Gamma),\notag
		\end{align}
		as $k\to+\infty$, for almost all $t\in(0,T)$. Therefore
		\begin{align*}
			\Theta_{k} f_{0}(\varphi_{\Theta_{k}}(x,t))\rightarrow0=\widetilde{\xi}(x,t)\quad\text{a.e. in }\big\{(x,t)\in Q_{T}:\,\widetilde{\varphi}(x,t)\in(-1,1)\big\}.
		\end{align*}
		On the other hand, if $\widetilde{\varphi}(x,t)=-1$ for some $x\in\overline{\Omega}$ and some $t\in(0,T)$ such that $\varphi_{\Theta_{k}}(t)\rightarrow\widetilde{\varphi}(t)$ strongly in $C(\overline{\Omega})$, then $f_{0}(\varphi_{\Theta_{k}}(x,t))\leq0$ for sufficiently large $k$ and therefore $\widetilde{\xi}(x,t)\leq0$, i.e., $\widetilde{\xi}(x,t)\in\partial I_{[-1,1]}(-1)$, almost everywhere on $\big\{\widetilde{\varphi}(x,t)=-1\big\}$. By the same argument, $\widetilde{\xi}(x,t)\geq0$, i.e., $\widetilde{\xi}(x,t)\in\partial I_{[-1,1]}(1)$, almost everywhere on $\big\{\widetilde{\varphi}(x,t)=1\big\}$. Hence, we obtain
		\begin{align*}
			\widetilde{\xi}\in\partial I_{[-1,1]}(\widetilde{\varphi})\quad\text{a.e. in }Q_{T}.
		\end{align*}
		Similarly, we can conclude that
		\begin{align*}
			\widetilde{\xi}_{\Gamma}\in\partial I_{[-1,1]}(\widetilde{\psi})\quad\text{a.e. on }\Sigma_{T}.
		\end{align*}
		Therefore, we complete the proof of Theorem \ref{doubleobstacle}.
\end{proof}

\appendix
\section{Appendix}
\setcounter{equation}{0}
\subsection{Useful tools}
\noindent We report some technical lemmas that have been used
in our analysis. First, we recall the compactness lemma of Aubin-Lions-Simon
type (see e.g., \cite{Simon})

\begin{lemma}
\label{ALS} Let $X_{0} \overset{c}{\hookrightarrow } X_{1}\subset X_{2}$
where $X_{j}$ are (real) Banach spaces ($j=0,1,2$). Let $1<p\leq +\infty $, $%
1\leq q\leq+ \infty ~$and $I$ be a bounded subinterval of $\mathbb{R}$. Then,
the sets
\begin{equation*}
\left\{ \varphi \in L^{p}\left( I;X_{0}\right) :\partial _{t}\varphi \in
L^{q}\left( I;X_{2}\right) \right\} \overset{c}{\hookrightarrow }
L^{p}\left( I;X_{1}\right),\quad \text{ if }1<p<+\infty,
\end{equation*}
and
\begin{equation*}
\left\{ \varphi \in L^{p}\left( I;X_{0}\right) :\partial _{t}\varphi \in
L^{q}\left( I;X_{2}\right) \right\} \overset{c}{\hookrightarrow } C\left(
I;X_{1}\right),\quad \text{ if }p=+\infty ,\text{ }q>1.
\end{equation*}
\end{lemma}

\begin{lemma}[Ehrling's Lemma]
\label{Ehrling}
Let $B_{0}$, $B_{1}$, $B_{2}$ be three Banach spaces. Assume that
  $B_{0}\subset B$ with compact injection and that $B\subset B_1$ with continuous injection. Then, for each $\epsilon>0$, there exists a positive constant $C_{\epsilon}$ depending on $\epsilon$ such that$$\Vert z\Vert_{B}\leq\epsilon\Vert z\Vert_{B_{0}}+C_{\epsilon}\Vert z\Vert_{B_{1}},\quad \forall\,z\in B_{0}.$$
\end{lemma}

\subsection{Proof of the extended \L ojasiewicz-Simon inequality}

The proof of Lemma \ref{LS} is an extension of \cite{AW} for the Cahn-Hilliard equation with singular potential and homogeneous Neumann boundary conditions and \cite{LamWu} for the bulk-surface Allen-Cahn system with regular potentials.

\begin{proof}[\textbf{Proof of Lemma \ref{LS}}]
Let $U$ be a (sufficiently) small neighborhood of $\boldsymbol{0}$ in $\mathcal{W}_{K,0}^{2}$. For any $\boldsymbol{u}\in U$, define  $\boldsymbol{\phi}:=\boldsymbol{u}+\boldsymbol{\varphi}_{\infty}\in\widetilde{U}=\{\boldsymbol{\varphi}_{\infty}\}+U$. By the strict separation property of $\boldsymbol{\varphi}_{\infty}$, we conclude that $\boldsymbol{\phi}$ stays uniformly away from pure states $\pm1$, this means that there exists a constant $\delta_{3}\in(0,1)$ such that
\begin{align}
	\Vert \phi\Vert_{L^{\infty}(\Omega)}\leq1-\delta_{3},\quad\Vert \phi_{\Gamma}\Vert_{L^{\infty}(\Gamma)}\leq1-\delta_{3},\quad\forall\,\boldsymbol{\phi}
	:=(\phi,\phi_{\Gamma})\in\widetilde{U}.\notag%\label{omegasepa}
\end{align}
Define the energy functional $\mathcal{E}:\,U\rightarrow\mathbb{R}$ as
\begin{align*}
	\mathcal{E}\big(\boldsymbol{u}\big):=E(\boldsymbol{\phi}),\quad\forall\,\boldsymbol{u}=\boldsymbol{\phi}-\boldsymbol{\varphi}_{\infty}\in U.
\end{align*}
Denote the first and second Fr\'echet derivatives of $\mathcal{E}$ as $\mathcal{M}$ and $\mathcal{L}$, respectively. By an argument similar to that in \cite{CFP}, we can conclude that $\mathcal{M}\in C(U,(\mathcal{W}_{K,0}^{2})')$ and for any $\boldsymbol{w}\in\mathcal{W}_{K,0}^{2}$, we have
\begin{align*}
	\mathcal{M}\big(\boldsymbol{u}\big)\big(\boldsymbol{w}\big)&=\int_{\Omega}\Big(\nabla\phi\cdot\nabla w+\beta(\phi)w+\pi(\phi)w\Big)\,\mathrm{d}x+\chi(K)\int_{\Gamma}(\phi_{\Gamma}-\phi)(w_{\Gamma}-w)\,\mathrm{d}S\\
	&\quad+\int_{\Gamma}\Big(\nabla_{\Gamma}\phi_{\Gamma}\cdot\nabla_{\Gamma}w_{\Gamma}
	+\beta_{\Gamma}(\phi_{\Gamma})w_{\Gamma}+\pi_{\Gamma}(\phi_{\Gamma})w_{\Gamma}\Big)\,\mathrm{d}S.
\end{align*}
In particular, since $\boldsymbol{u}\in U\subset\mathcal{W}_{K,0}^{2}$, we have $\mathcal{M}:U\subset\mathcal{W}_{K,0}^{2}\rightarrow\mathcal{L}_{(0)}^{2}$ and
\begin{align*}
	\mathcal{M}\big(\boldsymbol{u}\big)=\left(
	\begin{array}{c}
		\mathbf{P}_{\Omega}\big(-\Delta\phi+\beta(\phi)+\pi(\phi)\big)\\
		\mathbf{P}_{\Gamma}\big(\partial_{\mathbf{n}}\phi-\Delta_{\Gamma}\phi_{\Gamma}+\beta_{\Gamma}(\phi_{\Gamma})+\pi_{\Gamma}(\phi_{\Gamma})\big)
	\end{array}\right).
\end{align*}
Furthermore, we conclude that $\mathcal{L}\in C(U,\mathcal{B}(\mathcal{W}_{K,0}^{2},(\mathcal{W}_{K,0}^{2})'))$ and
\begin{align*}
	\mathcal{L}\big(\boldsymbol{u}\big)\big(\boldsymbol{w}\big)=\left(
	\begin{array}{c}
		\mathbf{P}_{\Omega}\big(-\Delta w+\beta'(\phi)w+\pi'(\phi)w\big)\\
		\mathbf{P}_{\Gamma}\big(\partial_{\mathbf{n}}w-\Delta_{\Gamma}w_{\Gamma}
		+\beta_{\Gamma}'(\phi_{\Gamma})w_{\Gamma}+\pi_{\Gamma}'(\phi_{\Gamma})w_{\Gamma}\big)
	\end{array}\right)\in\mathcal{L}_{(0)}^{2}\subset(\mathcal{W}_{K,0}^{2})'
\end{align*}
for any $\boldsymbol{u}\in U$, $\boldsymbol{w}\in\mathcal{W}_{K,0}^{2}$. Hence, $\mathcal{M}\in C^{1}(U,\mathcal{L}_{(0)}^{2})$ and $\mathcal{E}\in C^{2}(U;\mathbb{R})$.
Since $\boldsymbol{\varphi}_{\infty}$ satisfies \eqref{strong}, we obtain
\begin{align*}
	\mathcal{M}\big(\boldsymbol{0}\big)=\left(
	\begin{array}{c}
		\mathbf{P}_{\Omega}\big(-\Delta\varphi_{\infty}+\beta(\varphi_{\infty})+\pi(\varphi_{\infty})\big)\\
		\mathbf{P}_{\Gamma}\big(\partial_{\mathbf{n}}\varphi_{\infty}-\Delta_{\Gamma}\psi_{\infty}+\beta_{\Gamma}(\psi_{\infty})+\pi_{\Gamma}(\psi_{\infty})\big)
	\end{array}\right)=\boldsymbol{0},
\end{align*}
which infers that $\boldsymbol{0}$ is a stationary point of $\mathcal{E}$. Hence, it remains to show that there exist constants $\varsigma^{\ast}\in(0,1/2)$ and $b^{\ast}>0$ such that
$$\|\mathcal{M}(\boldsymbol{u})\|_{\mathcal{L}_{(0)}^{2}} \geq|\mathcal{E}(\boldsymbol{u})-\mathcal{E}(\boldsymbol{0})|^{1-\varsigma^{\ast}}$$
for all $\boldsymbol{u}\in\mathcal{W}_{K,0}^{2}$ satisfying $\|\boldsymbol{u}\|_{\mathcal{W}_{K,0}^{2}}\leq b^{\ast}$.

Define the linear operator $L:\mathcal{W}_{K,0}^{2}\rightarrow\mathcal{L}_{(0)}^{2}$ as $L:=\mathcal{L}\big(\boldsymbol{0}\big)$, i.e.,
\begin{align*}
	L\big(\boldsymbol{w}\big)=\left(
	\begin{array}{c}
		\mathbf{P}_{\Omega}\big(-\Delta w+\beta'(\varphi_{\infty})w+\pi'(\varphi_{\infty})w\big)\\
		\mathbf{P}_{\Gamma}\big(\partial_{\mathbf{n}}w-\Delta_{\Gamma}w_{\Gamma}
		+\beta_{\Gamma}'(\psi_{\infty})w_{\Gamma}+\pi_{\Gamma}'(\psi_{\infty})w_{\Gamma}\big)
	\end{array}\right),\quad\forall\,\boldsymbol{w}\in\mathcal{W}_{K,0}^{2}.
\end{align*}
For any two $\boldsymbol{w}$, $\boldsymbol{z}\in\mathcal{W}_{K,0}^{2}$
\begin{align*}
 \big(L(\boldsymbol{w}),\boldsymbol{z}\big)_{\mathcal{L}^{2}}&=\int_{\Omega}\nabla w\cdot\nabla z+\beta'(\varphi_{\infty})wz+\pi'(\varphi_{\infty})wz\,\mathrm{d}x\\
	&\quad+\int_{\Gamma}\nabla_{\Gamma}w_{\Gamma}\cdot\nabla_{\Gamma}z_{\Gamma}
	+\beta_{\Gamma}'(\psi_{\infty})w_{\Gamma}z_{\Gamma}+\pi_{\Gamma}'(\psi_{\infty})w_{\Gamma}z_{\Gamma}\,\mathrm{d}S\\
	&\quad+\chi(K)\int_{\Gamma}(z_{\Gamma}-z)(w_{\Gamma}-w)\,\mathrm{d}S\\
	&= \big(\boldsymbol{w},L(\boldsymbol{z})\big)_{\mathcal{L}^{2}}
\end{align*}
and we see that $L$ is a self-adjoint operator, i.e., $L=L^{*}$, where $L^{*}$ is the adjoint operator of $L$. Associated with $L$, we define the bilinear form $B(\boldsymbol{w},\boldsymbol{z})$ on $\mathcal{H}_{K,0}^{1}$ as follows:
\begin{align*}
	B(\boldsymbol{w},\boldsymbol{z})&=\int_{\Omega}\nabla w\cdot\nabla z+\beta'(\varphi_{\infty})wz+\pi'(\varphi_{\infty})wz\,\mathrm{d}x\\
	&\quad+\int_{\Gamma}\nabla_{\Gamma}w_{\Gamma}\cdot\nabla_{\Gamma}z_{\Gamma}
	+\beta_{\Gamma}'(\psi_{\infty})w_{\Gamma}z_{\Gamma}+\pi_{\Gamma}'(\psi_{\infty})w_{\Gamma}z_{\Gamma}\,\mathrm{d}S\\
	&\quad+\chi(K)\int_{\Gamma}(z_{\Gamma}-z)(w_{\Gamma}-w)\,\mathrm{d}S,\quad\forall\,\boldsymbol{w},\boldsymbol{z}\in\mathcal{H}_{K,0}^{1}.
\end{align*}
By $\mathbf{(A1)}$, $\mathbf{(A3)}$ , we can easily obtain the bilinear form
\begin{align*}
	B_{\lambda}\big(\boldsymbol{w},\boldsymbol{z}\big)=\lambda\big(\boldsymbol{w},\boldsymbol{z}\big)_{\mathcal{L}^{2}}+B\big(\boldsymbol{w},\boldsymbol{z}\big)
\end{align*}
is continuous and coercive in $\mathcal{H}_{K,0}^{1}$ for sufficiently large $\lambda$. Hence, by the Lax-Milgram theorem, $\lambda I+L$ is invertible and $(\lambda I+L)^{-1}:\mathcal{L}^{2}_{(0)}\rightarrow\mathcal{L}^{2}_{(0)}$ is compact. Then, we can obtain the Fredholm alternative result for the bulk-surface elliptic problem
\begin{align}
	\left\{
	\begin{array}{ll}
		-\Delta w+\beta'(\varphi_{\infty})w+\pi'(\varphi_{\infty})w=f,&\text{in }\Omega,\\
		K\partial_{\mathbf{n}}w=w_{\Gamma}-w,&\text{on }\Gamma,\\
-\Delta_{\Gamma}w_{\Gamma}+\partial_{\mathbf{n}}w+\beta_{\Gamma}'(\psi_{\infty})w_{\Gamma}+\pi_{\Gamma}'(\psi_{\infty})w_{\Gamma}=f_{\Gamma},&\text{on }\Gamma,
	\end{array}\right.\notag
\end{align}
where $\boldsymbol{f}:=(f,f_{\Gamma})\in\mathcal{L}^{2}_{(0)}$. Consequently, we obtain Rg$\big(L\big)=(\text{ker}(L^{*}))^{\perp}=(\text{ker}(L))^{\perp}$ and dimension of ker$(L)$ is finite. Let $\big\{\boldsymbol{\phi}_{j}\big\}_{j=1}^{N}:=\big\{(\phi_{j},\phi_{\Gamma,j})\big\}_{j=1}^{N}$ be the normalized orthogonal basis of Ker$(L)$ in $\mathcal{L}^{2}_{(0)}$ and $\big\{\boldsymbol{\phi}_{j}\big\}_{j=1}^{N}\in\mathcal{W}_{K,0}^{2}$ by the elliptic regularity theory. Let $\mathbb{P}\in\mathcal{B}(\mathcal{W}_{K,0}^{2},\mathcal{W}_{K,0}^{2})$ be the projection from $\mathcal{W}_{K,0}^{2}$ to Ker$(L)$, defined by
\begin{align*}
	\mathbb{P}\boldsymbol{v}=\sum_{j=1}^{N}(\boldsymbol{v},\boldsymbol{\phi}_{j})_{\mathcal{L}^{2}}\boldsymbol{\phi}_{j},\quad\forall\,\boldsymbol{v}\in\mathcal{W}_{K,0}^{2}.
\end{align*}
Then, the adjoint $\mathbb{P}^{\ast}\in\mathcal{B}\big((\mathcal{W}_{K,0}^{2})',(\mathcal{W}_{K,0}^{2})'\big)$ leaves $\mathcal{L}^{2}_{(0)}$ invariant. Indeed, for any $\boldsymbol{w}\in\mathcal{L}^{2}_{(0)}$ and $\boldsymbol{v}\in\mathcal{W}_{K,0}^{2}$,
\begin{align*}
	\big\langle \mathbb{P}^{\ast}\boldsymbol{w},\boldsymbol{v}\big\rangle_{(\mathcal{W}_{K,0}^{2})',\mathcal{W}_{K,0}^{2}}
	&=\big\langle\boldsymbol{w},\mathbb{P}\boldsymbol{v}\big\rangle_{(\mathcal{W}_{K,0}^{2})',\mathcal{W}_{K,0}^{2}}\\
	&=\big(\boldsymbol{w},\mathbb{P}\boldsymbol{v}\big)_{\mathcal{L}^{2}}\\
	&=\big(\boldsymbol{w},\sum_{j=1}^{N}(\boldsymbol{v},\boldsymbol{\phi}_{j})_{\mathcal{L}^{2}}\boldsymbol{\phi}_{j}\big)_{\mathcal{L}^{2}}\\
	&=\big(\sum_{j=1}^{N}(\boldsymbol{w},\boldsymbol{\phi}_{j})_{\mathcal{L}^{2}}\boldsymbol{\phi}_{j},\boldsymbol{v}\big)_{\mathcal{L}^{2}}\\
	&=\big\langle\sum_{j=1}^{N}(\boldsymbol{w},\boldsymbol{\phi}_{j})_{\mathcal{L}^{2}}\boldsymbol{\phi}_{j},\boldsymbol{v}\big\rangle_{(\mathcal{W}_{K,0}^{2})',\mathcal{W}_{K,0}^{2}}.
\end{align*}
Hence, we obtain $ \mathbb{P}^{\ast}\boldsymbol{w}=\sum_{j=1}^{N}(\boldsymbol{w},\boldsymbol{\phi}_{j})_{\mathcal{L}^{2}}\boldsymbol{\phi}_{j}\in\mathcal{L}^{2}_{(0)}$.

Now, we prove that $\mathcal{M}$ is analytic in a neighborhood $U$ of $\boldsymbol{0}$ in $\mathcal{W}_{K,0}^{2}$. Since $\beta$, $\beta_{\Gamma}$ are real analytic on $(-1,1)$ and $\pi$, $\pi_{\Gamma}$ are real analytic on $\mathbb{R}$, we can conclude that the mappings
\begin{align*}
	u\in\big\{L^{\infty}(\Omega):\,|u+\varphi_{\infty}|\leq1-\delta_{3}\text{ a.e. in }\Omega\big\}&\mapsto\beta(u+\varphi_{\infty})\in L^{\infty}(\Omega),\\
	u_{\Gamma}\in\big\{L^{\infty}(\Gamma):\,|u_{\Gamma}+\psi_{\infty}|\leq1-\delta_{3}\text{ a.e. on }\Gamma\big\}&\mapsto\beta_{\Gamma}(u_{\Gamma}+\psi_{\infty})\in L^{\infty}(\Gamma),\\
	u\in\big\{L^{\infty}(\Omega):\,|u+\varphi_{\infty}|\leq1-\delta_{3}\text{ a.e. in }\Omega\big\}&\mapsto\pi(u+\varphi_{\infty})\in L^{\infty}(\Omega),\\
	u_{\Gamma}\in\big\{L^{\infty}(\Gamma):\,|u_{\Gamma}+\psi_{\infty}|\leq1-\delta_{3}\text{ a.e. on }\Gamma\big\}&\mapsto\pi_{\Gamma}(u_{\Gamma}+\psi_{\infty})\in L^{\infty}(\Gamma),
\end{align*}
are analytic (in the sense of \cite[Definition 2.4]{HJ}). Then, by the embedding $$U\hookrightarrow\big\{\boldsymbol{u}:=(u,u_{\Gamma}):\,	\Vert u+\varphi_{\infty}\Vert_{L^{\infty}(\Omega)}, \Vert u_{\Gamma}+\psi_{\infty}\Vert_{L^{\infty}(\Gamma)}\leq1-\delta_{3}\big\},$$
it follows that $\mathcal{M}:\,U\rightarrow\mathcal{L}_{(0)}^{2}$ is analytic.
Finally, for the function spaces "$X, Y, V, W$" in \cite{Chill}, we can apply \cite[Corollary 3.11]{Chill} with $X=V=\mathcal{W}_{K,0}^{2}$, $W=Y=\mathcal{L}_{(0)}^{2}$ to derive the extended \L ojasiewicz-Simon inequality \eqref{5.2}.
\end{proof}

\medskip

\noindent \textbf{Acknowledgments.}
H. Wu is a member of the Key Laboratory of Mathematics for Nonlinear Sciences (Fudan University), Ministry
of Education of China. The research of H. Wu was partially supported by NNSFC Grant No. 12071084
and the Shanghai Center for Mathematical Sciences at Fudan University.

%\textbf{Compliance with Ethical Standards}
%
%\textbf{Conflict of Interest}: The authors declare that they have no
%conflict of interest.

\end{document}